\documentclass[10pt,a4paper]{article}%
\usepackage[centertags]{amsmath}
\usepackage{amsfonts}
\usepackage{amssymb}
\usepackage{amsthm}
\usepackage{epsfig}
\usepackage{setspace}
\usepackage{ae}
\usepackage{eucal}
\usepackage[usenames]{color}%
\usepackage{graphicx}

\theoremstyle{plain}
\newtheorem{thm}{Theorem}[section]

\newtheorem{cor}[thm]{Corollary}

\theoremstyle{definition}

\theoremstyle{remark}

\newtheorem{rem}[thm]{Remark}

\begin{document}

\title{Singular vorticity solutions of the incompressible Euler equation via inviscid limits}
\author{J\"org Kampen }
\maketitle

\begin{abstract}
Singular vorticity solutions of the incompressible $3D$-Euler equation are constructed which satisfy the BKM criterion (cf. \cite{MB}). The construction is done by inviscid limits of (essentially local) vorticity solutions of a family of related equations with  viscosity parameter, and where a damping potential term ensures the existence of a solution in a cone. Representations of (spatial derivatives of ) the solution of an extended Navier Stokes type viscosity system  in terms of convolutions with the Gaussian (with dispersion parameter $\nu$) and the first order spatial derivative of the Gaussian lead to an inviscid limit such the original vorticity function solves the Euler equation. For the inviscid limit  Lipschitz regularity of the convoluted terms of that representations is used and proved  by an iterative scheme for some class of strongly regular data. Such inviscid limit vorticity solutions of the incompressible Euler vorticity equation can become singular at a point of the boundary of a finite domain (the tip of that cone). 
 The construction can be applied locally within compressible Navier Stokes equations systems of dimension $D\geq 3$ with degeneracies of viscosity and Lam\'{e} viscosity such that some type of turbulent dynamics occurs locally as the Reynold numbers become very large. Some solution branches of the Euler equation loose regularity at any local time while others can be continued.
\end{abstract}


\section{Singular vorticity of the Euler equation via two hypotheses concerning limits of well-defined solutions of Navier Stokes type equations }
We are concerned with singular vorticity solutions of the Cauchy problem of the incompressible Euler equation in dimension $D=3$, i.e., the Cauchy problem for the velocity $v_i,~1\leq i\leq D$ of the form
\begin{equation}
\left\lbrace \begin{array}{ll}
\frac{\partial v_i}{\partial t}+\sum_{j=1}^D v_j\frac{\partial v_i}{\partial x_j}=-\frac{\partial p}{\partial x_i},\\
\\
\sum_{i=1}^D\frac{\partial v_i}{\partial x_i}=0,\\
\\
v_i(0,x)=h_i(x),~\mbox{for}~1\leq i\leq D,
\end{array}\right.
\end{equation}
where the problem is posed on the $D$-dimensional torus of on the whole space, i.e., $h_i\in H^s\left({\mathbb T}^D\right)$, or $h_i\in H^s\left({\mathbb R}^D\right)$ for some $s\geq 0$, and $i\in \left\lbrace  1,2,3\right\rbrace $, (i.e., we have finite energy at least). Recall that this equation determines a Cauchy problem for the vorticity
\begin{equation}
\omega =\mbox{curl}(v)=\left(\frac{\partial v_3}{\partial x_2}-\frac{\partial v_2}{\partial x_3},\frac{\partial v_1}{\partial x_3}-\frac{\partial v_3}{\partial x_1},\frac{\partial v_2}{\partial x_1}-\frac{\partial v_1}{\partial x_2} \right)
\end{equation}
of the form
\begin{equation}\label{vorticity}
\frac{\partial \omega}{\partial t}+v\cdot \nabla \omega=\frac{1}{2}\left(\nabla v+\nabla v^T\right)\omega, 
\end{equation}
which has to be solved with the initial data $\mbox{curl}(h)$ at time $t=0$, and where $h=(h_1,\cdots,h_D)^T$.
The following considerations can be applied for the domain of a torus ${\mathbb T}^D$ of dimension $D=3$ and for the whole space of ${\mathbb R}^D$, where there are slight differences with respect to the coordinate transformation and with respect to the formulation of the Biot-Savart law. The following description is with respect to the whole space ${\mathbb R}^D$. We use the dimension parameter $D$ in order to indicate that $D\geq 3$ seems essential, but we stick to $D=3$ for simplicity, as the vorticity or curl are well-known classical objects in this case.  
In coordinates the vorticity equation is
\begin{equation}\label{vort}
 \frac{\partial \omega_i}{\partial t}+\sum_{j=1}^3v_j\frac{\partial \omega_i}{\partial x_j}=\sum_{j=1}^3\frac{1}{2}\left(\frac{\partial v_i}{\partial x_j}+\frac{\partial v_j}{\partial x_i} \right)\omega_j.
\end{equation}
The vorticity equation form of the Navier Stokes equation has an additional viscosity term $\nu \Delta \omega$ along with viscosity $\nu >0$ (for each component equation $1\leq i\leq D$). As we have (cf. \cite{MB})
\begin{equation}\label{vel}
v(t,x)=\int_{{\mathbb R}^3}K_3(x-y)\omega(t,y)dy,~\mbox{where}~K_3(x)h=\frac{1}{4\pi}\frac{x\times h}{|x|^3},
\end{equation}
we may interpret (\ref{vort}) as a nonlinear partial integro-differential equation for the vorticity. In case of the torus a corresponding relation may be formulated via Fourier representations.  In the following we use the abbreviation
\begin{equation}\label{abbrev}
w^{\int}(t,.):=\int_{{\mathbb R}^3}K_3(.-y)w(t,y)dy=\int_{{\mathbb R}^3}K_3(y)w(t,.-y)dy
\end{equation}
for any function $w:{\mathbb R}^3\rightarrow {\mathbb R}$ such that (\ref{abbrev}) is well-defined. Using the convolution rule as in the last equation in (\ref{abbrev}) is useful, since this allows us to have simple expressions of derivatives of the velocity components (without taking derivatives of the kernel $K_3$). If we consider a vorticity Navier Stokes equation with an additional viscosity operator $\nu \Delta$ and on a different time scale (defined by a time transformation and indicated by some parameter $\rho >0$ of this time-transformation), then we shall denote the corresponding solution function by $\omega^{\rho,\nu}_i,~1\leq i \leq D$, and correspondingly to (\ref{vel}) we write
\begin{equation}\label{vel*}
v^{\rho,\nu}(\tau,x)=\int_{{\mathbb R}^3}K_3(y)\omega^{\rho,\nu}(\tau,x-y)dy.
\end{equation}
Similar relations hold for the torus with an operator defined via Fourier transformations, of course. Here, we mention the convolution rule in order to remind us that derivatives of velocities have a simple relation to derivatives of vorticity. In general, we write $\omega_{i,jk}:=\frac{\partial^2}{\partial x_j\partial x_k}\omega_i$ as usual, and similarly for velocities. We may also suppress the arguments and write $w^{\int}$ for a transformation of a function $w$ as in (\ref{abbrev}). 
For smooth functions $u_i:[0,\infty)\times \Omega\rightarrow {\mathbb R},~1\leq i\leq 3 $, along with some $\Omega\subseteq {\mathbb R}^3$ we consider the ansatz
\begin{equation}\label{ansatz}
\omega_i(t,x)=\frac{u_i\left(\frac{t}{\sqrt{1-t^2}},(1-t)\frac{2}{\pi}\arctan(x)\right) }{1-t}
\end{equation}
based on the coordinate diffeomorphism with the domain $D_1:=[0,1)\times {\mathbb R}^3$ for the $t$- and $x$-coordinates, i.e., the domain of the function $\omega_i,~1\leq i\leq D$. Here, we write $\arctan(x)=\left(\arctan(x_1),\cdots ,\arctan (x_n) \right)^T$. For $(t,y)=\left( t,(1-t)\frac{2}{\pi}\arctan(x)\right) $ we have $y_j\in (-(1-t),(1-t))$ such that we have a cone image 
\begin{equation}
\begin{array}{ll}
K_0:=\cup_{s,~0\leq t(s)< 1} K_{t(s)}:=[0,\infty)\times \cup_{0\leq t(s)< 1}\left(-1+t(s),1-t(s)\right) ^3\\
\\
\subset [0,\infty)\times {\mathbb R}^3
\end{array}
\end{equation}
for $s$- and $y$-coordinates of the form
\begin{equation}\label{coordtrans}
\left(s,y \right)=\left(\frac{t}{\sqrt{1-t^2}},(1-t)\frac{2}{\pi}\arctan(x) \right).
\end{equation}
Here $t(s)$ is the inverse of $s=s(t)$ given as in (\ref{coordtrans}). 
In the following the cone in original time coordinates $t$ is denoted by $K^t_0$. 
Alternatively, and also on (a suitable realization of a) torus  we could use a similar diffeomorphism with $\left(s,y \right)=\left(\frac{t}{\sqrt{1-t^2}},(1-t^2)\frac{x}{1+\frac{1}{2}|x|^2} \right)$ for example.
As a matter of elegance in notation below, the Cauchy problems considered in transformed coordinates can be extended trivially, where the solution function is identical to the zero function outside the cone. This is true if the local-time solution function of the original Navier Stokes equation or related extended viscosity system has decay to zero at spatial infinity. The complementary set of the cone $K_0$ is a subset of the whole domain $D_{\infty}:=[0,\infty)\times {\mathbb R}^3$. Regular extension with the zero function of the viscosity limit turns out to be possible for problems of physical interest, i.e., for data of polynomial decay at infinity. This is not essential but it is convenient.
We write $t(s)$ and $x(s,y)$ for the components of the reverse.  In the following for all $1\leq j\leq 3$ the function $(t,x)\rightarrow (1-t)\frac{2}{\pi}\arctan(x_j)$ is denoted by $b_j$ (the same notation may be used for the alternative transformation with $(1-t)\frac{x}{1+\frac{1}{2}|x|^2}$, and also be used for the torus). As usual, we denote the derivative of $b_j$ with respect to the spatial variable $x_k$ by $b_{j,k}$ and the derivative of $b_j$ with respect to $s$ by $b_{j,s}$ etc.. All these terms are products of time functions and spatial functions. For this reason we also denote $b_j=(1-t)b_j^{s}$ etc., extracting the spatial part $b^{s}_j$ of the coeefficient $b_j$ if this is convenient. We shall consider solution functions $(s,y)\rightarrow u^{\nu}_i(s,y)$ of Navier Stokes type equations which approximate $u_i(s,y)=u_i\left(\frac{t}{\sqrt{1-t^2}},(1-t)\frac{2}{\pi}\arctan(x)\right) $ in (\ref{ansatz}) above for small viscosity $\nu>0$.  Note that
\begin{equation}
\frac{\partial \omega_i}{\partial t}=\frac{u_i(s,x)}{(1-t)^2}+\frac{u_{i,s}}{1-t}\frac{ds}{dt}+\sum_{j=1}^3\frac{u_{i,j}}{(1-t)}\frac{\partial y_j}{\partial t},
\end{equation}
where $\frac{\partial s}{\partial t}=\frac{1}{\sqrt{1-t^2}}+t\frac{-1}{2}\sqrt{1-t^2}^{-3}(-2t)=\frac{1-t^2+t^2}{\sqrt{1-t^2}^3}=\frac{1}{\sqrt{1-t^2}^3}$. Note that
\begin{equation}
b_{j,t}=\frac{\partial y_j}{\partial t}=-\frac{2}{\pi}\arctan(x_j)
\end{equation}
are bounded smooth coefficients with bounded derivatives. For the convenience of the reader we derive the equation for 
\begin{equation}
u_i(s,y)=(1-t)\omega_i(t,x)
\end{equation}
explicitly. If we consider the related vorticity form of the Navier Stokes equation we add a superscript $\nu>0$ and write
\begin{equation}
u^{\nu}_i(s,y)=(1-t)\omega^{\nu}_i(t,x),
\end{equation}
where $\omega^{\nu}_i,~1\leq i\leq D$ satisfies the vorticity form of the Navier Stokes equation with viscosity $\nu$, i.e., the Laplacian term $-\nu\Delta \omega$ is added on the side of the equation with the time derivative, i.e., $-\Delta \omega_i$ is added for each component $1\leq i\leq D$ of the equation. We identify for all $s,y$ in the considered domain
\begin{equation}
u_i(s,y)=u^0_i(s,y)=\lim_{\nu\downarrow 0}u^{\nu}_i(s,y).
\end{equation}
Similar for velocity components: we write $v^{\nu}_i$ in order to indicate that these are velocity components of a solution function for an incompressible Navier Stokes equation with viscosity $\nu >0$.
The additional symbol $u_i$ which reminds us that we consider an equation which is trivially extended to the whole domain in transformed coordinates (there is no other reason to introduce this additional symbol but this reminder). We write the transformation in a general form, where the fact that we use a 'diagonal' transformation simplifies the equation as indicated by the Kronecker-$\delta$s' symbolized by $\delta_{ik}$ for $1\leq i,k\leq D$. This way it is easy to transfer the following considerations to alternative transformations which are not diagonal but better suited for other domains and general compressible Navier Stokes equation systems considered in the next section.
We have
\begin{equation}
\omega_{i,j}=\frac{1}{1-t}\sum_{k}u_{i,k}b_{k,j}\delta_{kj},
\end{equation}
which leads to
\begin{equation}
\Delta \omega_i=\frac{1}{1-t}\sum_{j}\left(\sum_{k,m}u_{i,k,m}b_{k,j}b_{m,j}\delta_{kj}\delta_{mj}+\sum_{k,m}u_{i,k}b_{k,j,m}\delta_{kj}\delta_{mj} \right).
\end{equation}
Next for the Burgers term we have a measure change according to the spatial part of the transformation, i.e., 
\begin{equation}
dw=\Pi_{k=1}^ndw_k=(1-t)^n\frac{2^n}{\pi^n}\Pi_{k=1}^n\frac{1}{1+x_k^2}dx_k 
\end{equation}
and a transformed kernel $K_3(y)=K^*_3(z)$ (according to the spatial part of the transformation and with notation $K_3\omega_j=(K_3\omega)_j$) such that
\begin{equation}
\sum_{j=1}^3\left( \int_{{\mathbb R}^3}K_3(y)\omega_j(t,x-y)dy\right) \frac{\partial \omega_i}{\partial x_j}
\end{equation}
becomes
\begin{equation}
\frac{1}{(1-t)^2}\sum_{j=1}^3\left( \int_{K_0}K^*_3(w)u_j(s,y-w)\frac{\pi^3}{2^3}\frac{\Pi_{k=1}^n(1+x_k^2)}{(1-t)^n}dw\right) \sum_{k}u_{i,k}b_{k,j}\delta_{kj}.
\end{equation}
Mote that the explicit form of the transformed kernel is given via the inverse of the coordinate transformation in (\ref{coordtrans}) above, and can be made explicit straightforwardly.
Similarly, the Leray projection term
\begin{equation}
\sum_{j=1}^3\frac{1}{2}\left(\frac{\partial v_i}{\partial x_j}+\frac{\partial v_j}{\partial x_i} \right)\omega_j
\end{equation}
becomes
\begin{equation}
\begin{array}{ll}
\frac{1}{(1-t)^2}\sum_{j=1}^3\frac{1}{2}{\Bigg (}\int_{K_0}K^*_3(w)\sum_{k}u_{i,k}(s,y-w)\frac{\Pi_{k=1}^n(1+x_k^2)}{(1-t)^n}b_{k,j}\delta_{jk}dw\\
\\
\hspace{2cm}+\int_{K_0}K^*_3(w)\sum_{k}u_{j,k}(s,y-w)\frac{\Pi_{k=1}^n(1+x_k^2)}{(1-t)^n}b_{k,i}\delta_{ik}dw {\Bigg )}u_j.
\end{array}
\end{equation}
Note that the small volume of the cone $K_0$ (where the $u_j$ are supported) will compensate for the singular factor $\frac{1}{(1-t)^n}$. Note furthermore that all occurrences of $x_j$ and $t$ in the equation for the function $u_j,~1\leq j\leq D$ are always understood as functions $x_j=x_j(s,y_j)$ and $t=t(s)$ for the sake of brevity.
Hence, the vorticity form of the Navier Stokes equation in dimension $D=3$
\begin{equation}\label{vort2}
 \frac{\partial \omega^{\nu}_i}{\partial t}-\nu \Delta \omega^{\nu}_i+\sum_{j=1}^3v_j\frac{\partial \omega^{\nu}_i}{\partial x_j}=\sum_{j=1}^3\frac{1}{2}\left(\frac{\partial v^{\nu}_i}{\partial x_j}+\frac{\partial v^{\nu}_j}{\partial x_i} \right)\omega^{\nu}_j
\end{equation}
gets, literally written, the form
\begin{equation}
\begin{array}{ll}
\frac{u^{\nu}_i(s,x)}{(1-t)^2}+\frac{u^{\nu}_{i,s}}{1-t}\frac{1}{\sqrt{1-t^2}^3}+\sum_{j=1}^3\frac{u^{\nu}_{i,j}}{(1-t)}b_{j,t}\\
\\
-\nu \frac{1}{1-t}\sum_{j}\left(\sum_{k,m}u^{\nu}_{i,k,m}b_{k,j}b_{m,j}\delta_{jk}\delta_{mj}+\sum_{k,m}u^{\nu}_{i,k}b_{k,j,m}\delta_{kj}\delta_{mj} \right)\\
\\
+\frac{1}{(1-t)^2}\sum_{j=1}^3\left( \int_{K_0}K^*_3(w)u^{\nu}_j(s,y-w)d\frac{\Pi_{k=1}^n(1+x_k^2)}{(1-t)^n}w\right) \sum_{k}u^{\nu}_{i,k}b_{k,j}\delta_{kj}\\
\\
=\frac{1}{(1-t)^2}\sum_{j=1}^3\frac{1}{2}{\Bigg (}\int_{K_0}K^*_3(w)\sum_{k}u^{\nu}_{i,k}(s,y-w)b_{k,j}\delta_{jk}\frac{\Pi_{k=1}^n(1+x_k^2)}{(1-t)^n}dw\\
\\
\hspace{2.5cm}+\int_{K_0}K^*_3(w)\sum_{k}u^{\nu}_{j,k}(s,y-w)b_{k,i}\delta_{ik}\frac{\Pi_{k=1}^n(1+x_k^2)}{(1-t)^n}dw {\Bigg )}u^{\nu}_j.
\end{array}
\end{equation}
Using $b_i=(1-t)b^s_i$ and similar relations for spatial derivatives and ordering terms a bit we get
\begin{equation}
\begin{array}{ll}
u^{\nu}_{i,s}+\sqrt{1-t^2}^3\sum_{j=1}^3u^{\nu}_{i,j}b_{j,t}\\
\\
-\nu \sqrt{1-t^2}^3\sum_{j}\left(\sum_{k,m}u^{\nu}_{i,k,m}b_{k,j}b_{m,j}\delta_{kj}\delta_{mj}+\sum_{k,m}u^{\nu}_{i,k}b_{k,j,m}\delta_{kj}\delta_{mj} \right)\\
\\
+\sqrt{1-t^2}^3\sum_{j=1}^3\left( \int_{K_0}K^*_3(w)u^{\nu}_j(s,y-w)\frac{\Pi_{k=1}^n(1+x_k^2)}{(1-t)^n}dw\right) \sum_{k}u_{i,k}b^s_{k,j}\delta_{kj}\\
\\
=\sqrt{1-t^2}^3\sum_{j=1}^3\frac{1}{2}{\Bigg (}\int_{K_0}K^*_3(w)\sum_{k}u^{\nu}_{i,k}(s,y-w)b^s_{k,j}\delta_{jk}\frac{\Pi_{k=1}^n(1+x_k^2)}{(1-t)^n}dw\\
\\
\hspace{2.5cm}+\int_{K_0}K^*_3(w)\sum_{k}u^{\nu}_{j,k}(s,y-w)b^s_{k,i}\delta_{ik}\frac{\Pi_{k=1}^n(1+x_k^2)}{(1-t)^n}dw {\Bigg )}u^{\nu}_j\\
\\
-\frac{\sqrt{1-t^2}^3}{(1-t)}u^{\nu}_i(s,y).
\end{array}
\end{equation}
The additional terms with the coefficient $\nu$ correspond to an additional viscosity term $\nu \Delta \omega$ along with viscosity $\nu >0$ of the vorticity form of the Navier Stokes equation. As we have (cf. \cite{MB})
\begin{equation}
v^{\nu}(t,x)=\int_{{\mathbb R}^3}K_3(x-y)\omega^{\nu}(t,y)dy,~\mbox{where}~K_3(x)h=\frac{1}{4\pi}\frac{x\times h}{|x|^3},
\end{equation}
the velocity data $v^{\nu}_i$ of the original Navier Stokes equation can be recovered from the vorticity equation. For the sake of simplification we may write
\begin{equation}
u^{\nu,\int}_{i}=\int_{K_0}K^*_3(w)u^{\nu}_i(.,.-w)\frac{\Pi_{k=1}^n(1+x_k^2)}{(1-t)^n}dw.
\end{equation}
Then in the $s$- and $y$-coordinates and for $u_i(s,y)=(1-t)\omega_i(t,x)$ on the domain $D_{\infty}$ (trivially extended by values zero outside the image cone) the vorticity form of the Navier Stokes equation in dimension $D=3$ (cf. \cite{MB}) gets the form (using the fact that we have a diagonal transformation)
\begin{equation}\label{transformed}
\begin{array}{ll}
u^{\nu}_{i,s}+
\sqrt{1-t(s)}^3\sum_jb_{j,t}u^{\nu}_{i,j}+\sqrt{1-t(s)}^3
\sum_{j}b^s_{j,j}u^{\nu,\int}_{j}u_{i,k}+\frac{\sqrt{1-t^2(s)}^3}{1-t}u^{\nu}_i\\
\\
-\nu \sqrt{1-t(s)}^3\sum_{j} b_{j,j}b_{j,j}u^{\nu}_{i,j,j}-\nu \sqrt{1-t(s)}^3\sum_{j} b_{j,j,j}u^{\nu}_{i,j}\\
\\
=\sqrt{1-t(s)}^3\sum_j\frac{1}{2}\left( (b^s_{j,j} u^{\nu,\int}_{i,j})+(b^s_{i,i} u^{\nu,\int}_{j,i})\right)u^{\nu}_j.
\end{array}
\end{equation}
As we are only interested in the viscosity limit we use an alternative viscosity extension where we add $\nu$ times the Laplacian to the transformed equation. Accordingly let $u^{\nu,-}_{i},~1\leq i\leq D$ be the solution of the related problem
\begin{equation}\label{transformed-}
\begin{array}{ll}
u^{\nu,-}_{i,s}+
\sqrt{1-t(s)}^3\sum_jb_{j,t}u^{\nu}_{i,j}+\sqrt{1-t(s)}^3
\sum_{j}b^s_{j,j}u^{\nu,,-\int}_{j}u^{\nu,-}_{i,k}+\frac{\sqrt{1-t^2(s)}^3}{1-t}u^{\nu,-}_i\\
\\
-\nu \Delta u^{\nu,-}_{i}
=\sqrt{1-t(s)}^3\sum_j\frac{1}{2}\left( (b^s_{j,j} u^{\nu,-,\int}_{i,j})+(b^s_{i,i} u^{\nu,-,\int}_{j,i})\right)u^{\nu,-}_j,
\end{array}
\end{equation}
where $u^{\nu,-}_{i}(0,.)=\tilde{h}_i(.),~1\leq i\leq D$ with the transformed data
 $\tilde{h}_i(y)=h_i(x)$ according to the spatial part of the transformation. Note that for all $1\leq i\leq D$ $\tilde{h}_i$ is supported on the domain $[-1,1]^n$. However 
we consider this problem on the whole domain, i.e., we search for solutions 
$u_i^{\nu,-,e},~1\leq i\leq D$ of (\ref{transformed-}) on the domain $[0,\infty)\times {\mathbb R}^D$ with
\begin{equation}\label{data-} 
u^{\nu,-,e}_i:[0,\infty)\times {\mathbb R}^D\rightarrow {\mathbb R},~u^{\nu,-,e}_i(0,.)=h^e_i(.),
\end{equation}
where for all $1\leq i\leq D$ the functions $h^e_i$ belong to a class of functions with
\begin{equation}\label{data-1}
h^e_i:{\mathbb R}^D\rightarrow {\mathbb R},~ h^e_i(x)= 0~\mbox{ iff }x\in {\mathbb R}^D\setminus (-1,1)^D.  
\end{equation}
The construction is such that the viscosity limit of the constructed solution functions of the problem in (\ref{transformed-}) considered on the whole domain with (\ref{data-}) and (\ref{data-1}) is directlly realted to a solution of a transformed regular part of the Euler equation. The extension of the transformed equation to the whole domain has the advantage that classical representations of $u^{\nu,-,e}_{i},~1\leq i\leq D$ in terms of convolutions with the Gaussian defined on the whole space are available, and this simplifies some estimates.  Therefore we mainly work with $u^{\nu,-,e}_{i},~1\leq i\leq D$ below and then transfer some estimates to the viscosity limits via compactness arguments for some subsequence $(u^{\nu_k,e,-}_{i})_k,~1\leq i\leq D,~ k\geq 1$, where $\nu_k\downarrow 0$. Sometimes we add some remarks for the latter transformed local solutions of the Navier Stokes equations.
Note that the potential term in (\ref{transformed-}) has not a real singularity as $\frac{\sqrt{1-t^2}^3}{1-t}=(1+t)\sqrt{1-t^2}$, and for the function $t(\sigma)$ we have $t\left(\left[0,\infty\right)   \right)\subseteq \left[0,1\right) $ such that all other coefficients are bounded with respect to time. 
Note furthermore, that spatial coefficient notation with upper script $s$ is applied for the nonlinear terms in order to observe that the coefficients are bounded.

For a specific parameter $\nu$ a solution of (\ref{transformed-}) in time $s$-coordinates denoted by $u^{\nu,-}_i,~1\leq i\leq 3$ is related to a solution $u^{\nu,-,*}_i,~1\leq i\leq 3$ in the corresponding $t$-time coordinates,  such that
\begin{equation}
u^{\nu,-}_i(s,y)=u^{\nu,-,*}_i(t,y)=u^{\nu,-,*}_i(t(s),y)
\end{equation}
on the domains $D_{\infty}$ , i.e., the domain with respect to $s$-coordinates, and $D_1$, i.e., the corresponding domain with respect to $t$-coordinates (cf. definitions below). Note that
\begin{equation}
u^{\nu,-}_{i,s}=u^{\nu,-,*}_{i,t}\sqrt{1-t^2}^3,~\mu(t):=\frac{ds}{dt}
\end{equation}
where we recall $\mu(t):=\frac{ds}{dt}=\frac{1}{\sqrt{1-t^2}^3}$
In this original time coordinates the equation in (\ref{transformed-}) becomes another equation for $u^{\nu,-,*}_{i},~1\leq i\leq D$, where
\begin{equation}\label{transformed-**}
\begin{array}{ll}
u^{\nu,-,*}_{i,t}+
\sum_jb_{j,t}u^{\nu}_{i,j}+
\sum_{j}b^s_{j,j}u^{\nu,-,*,\int}_{j}u^{\nu,-,*}_{i,k}\\
\\
-\nu \mu(t)\Delta u^{\nu,-,*}_{i}
=\sum_j\frac{1}{2}\left( (b^s_{j,j} u^{\nu,-,*,\int}_{i,j})+(b^s_{i,i} u^{\nu,-,*,\int}_{j,i})\right)u^{\nu,-,*}_j\\
\\
-\frac{1}{1-t}u^{\nu,-,*}_i,
\end{array}
\end{equation}
where the initial data remain the same. If we extend to the whole domain we write $u^{\nu,-,e,*}_i,~1\leq i\leq D$ accordingly  with $u^{\nu,-,e,*}_i(0,.)=h^e_i(.)$. Note that the solution of the equation in (\ref{transformed-**}) is supported on the domain
\begin{equation}
\begin{array}{ll}
K^t_0=\cup_{0\leq t\leq 1} K^t_{t}:=[0,1)\times \cup_{0\leq t\leq 1}\left(-1+t,1-t\right) ^3,
\end{array}
\end{equation}
such that the solution function exists, i.e., the singularity $\frac{1}{1-t}$ of the potential term and of the diffusion term do not affect the integrability of the solution function at the tip of the cone $K^t_0$.
In the following the symbol $\overline{\lim}$ refers to the limes superior.
In the following we assume regular data $h_i\in H^s\cap C^{\infty}$ for $s\geq 3$ for simplicity, because the main motivation is to show that singularities can evolve from regular data. It seems that this assumption can be weakened but it simplifies local contraction results and the consideration of viscosity limits.
We have
\begin{thm}\label{hyposing1}
We assume regular data $h_i\in H^s\cap C^{\infty}$ for $s\geq 3$ for $1\leq i\leq D$.
For all $\nu >0$ there exists a local regular bounded solution $u^{\nu,-,*}_i,~1\leq i\leq 3:\left[0,1\right)\times {\mathbb R}^3\rightarrow {\mathbb R}$ corresponding to a global regular bounded solution $u^{\nu.-}_i,~1\leq i\leq 3:\left[0,\infty\right)\times {\mathbb R}^3\rightarrow {\mathbb R}$ of equation (\ref{transformed-}), and such that there is a well-defined viscosity limit function in $C^{1,2}$ in transformed coordinates as $\nu \downarrow 0$. Furthermore, if 
\begin{equation}\label{ass}
\lim_{\nu\downarrow 0}\overline{\lim}_{s\uparrow \infty}u^{\nu,-}_i(s,0)=\lim_{\nu\downarrow 0}\overline{\lim}_{t\uparrow 1} u^{\nu,-,*}_i(t,0)=c\neq 0
\end{equation}
for some $1\leq i\leq 3$ and some constant $c\neq 0$, then the function $\omega_i,~1\leq i\leq 3$ with
\begin{equation}
\omega_i(t,x):=\lim_{\nu\downarrow 0}\frac{u^{\nu,-,*}_i(t,y)}{1-t},~1\leq i\leq 3,
\end{equation}
with $y=y(t,x)=(1-t)\frac{2}{\pi}\arctan(x)$ is well-defined on $D_1$, and satisfies the incompressible vorticity form of the Euler equation. Furthermore, this function $\omega_i,~1\leq i\leq D$ becomes singular at the tip $(1,0)$ of the cone $K^t_0$. 
\end{thm}
The assumption in (\ref{ass}) of the preceding theorem \ref{hyposing1} can be cancelled if we consider an additional local time-transformation of the  Euler equations. For $\rho>0$ we consider a family of coordinate diffeomorphisms from the domain $D_{\rho}:=[0,\rho)\times {\mathbb R}^3$ of the $t\in [0,\rho)$- and $x$-coordinates to the cone in the domain $D_{\infty}:=[0,\infty)\times {\mathbb R}^3$ of the $\sigma$- and $y^{\rho}$-coordinates of the form
\begin{equation}\label{coord}
\left(\sigma,y^{\rho} \right)=\left(\frac{t}{\sqrt{\rho^2- t^2}},\left( \rho-t\right) \frac{2}{\pi}\arctan(x)\right).
\end{equation}
Define
\begin{equation}
\omega^{\rho,\nu,-}_i(\sigma,x):=\omega^{\nu,-} (t,x)=:\frac{u^{\rho,\nu,-}_i\left(\sigma,y^{\rho} \right) }{\rho-t},
\end{equation}
i.e., for specific parameters $\nu >0$ and $\rho>0$ a solution of the transformation in $\sigma$- and $y^{\rho}$-coordinates of the equation (\ref{transformed}) is denoted by $u^{\rho,\nu,-}_i,~1\leq i\leq 3$ and $u^{\rho, \nu,-,*}_i,~1\leq i\leq 3$ in the respective coordinates. Analogously as in the case $\rho=1$ we consider
\begin{equation}
u^{\rho,\nu,-}_i(\sigma,y^{\rho})=u^{\rho,\nu,-,*}_i(t(\sigma),y^{\rho})
\end{equation}
on the respective domains $D_{\infty}$ and $D_{\rho}$ as solution functions of certain Cauchy problems.  Similar for the related function $u^{\rho,\nu',-,*}_i,~1\leq i\leq D$ which is defined analogously as in the case $\rho=1$ above (cf. the discussion above).
Note that for each $\rho>0$ in $\left( \sigma,y^{\rho}\right)$-coordinates  solution functions $u^{\rho,\nu,-}_i,~1\leq i\leq 3$ have support on the cone image, i.e., the cone image which corresponds to the cone $K^{\rho}_0$ (with $K^{\rho}_0$ analogously defined as $K_0$) in $\left(\sigma,y^{\rho} \right)$-coordinates, and are considered to be trivially extended outside. Again, the cone image in original $t$-coordinates is denoted by $K^{\rho,t}_0$. 
\begin{thm}\label{hyposing2}
There is a class of regular data $S$ which is dense in $\left[ L^2\left({\mathbb R}^3\right)\right]^3$ such that for all $\nu >0$ and $\rho\in (0,1)$ there  exists a local regular bounded solution $u^{\rho,\nu,-*}_i,~1\leq i\leq 3:D_{\rho}\rightarrow {\mathbb R}$ which corresponds to a global bounded solution $u^{\rho,\nu,-}_i,~1\leq i\leq 3:\left[0,\infty\right)\times {\mathbb R}^3\rightarrow {\mathbb R}$ of the equation (\ref{transformed-}) for a tuple of data $h_i, 1\leq i\leq 3$ in $S$. Furthermore, for all data tuples $h_i(0,.),~1\leq i\leq 3$ in $S$  there exists a time horizon $\rho>0$ such that the function $\omega_i,~1\leq i\leq 3$ with
\begin{equation}
\omega_i(t,x)=\lim_{\nu\downarrow 0}\frac{u^{\rho,\nu,-,*}_i(t(\sigma),y^{\rho})}{\rho-t},~1\leq i\leq 3,
\end{equation}
satisfies the incompressible vorticity form of the Euler equation on $D_{\rho}$, and is singular at the tip of the cone $K^{\rho,t}_0$, i.e., at $(\rho,0)$. 
\end{thm}

\begin{rem}
The connection to the preceding theorem is that for elements of a dense set of data in $L^2$ there some  $\rho>0$ such that there is a viscosity limit $\nu \downarrow 0$ in transformed coordinates such that 
\begin{equation}\label{hypo2}
\overline{\lim}_{\sigma\uparrow \infty}{\big |}u^{\rho,\nu,-}_i(\sigma,0){\big |}=\overline{\lim}_{t\uparrow \rho}{\big |}u^{\rho,\nu,-,*}_i(t(\sigma),0){\big |}= c\neq 0
\end{equation}
for some $1\leq i\leq 3$ and some constant $c\neq 0$. Similarly for the function $u^{\rho,\nu,-,*}_i, 1\leq i\leq D$. This implies the existence of a singular solution to the vorticity equation.
\end{rem}

\begin{proof}
As the functions space $C^{\infty}_0$ of smooth data which decay to zero at spatial infinity is dense in $L^2$ in section 3 below we shall consider a dense subclass in $C^{\infty}_0$ and construct a time-local solution of the related Cauchy problems for $u^{\rho,\nu',-,*}_i,~1\leq i\leq D$. The construction of local time solutions (resp. global time solutions) $u^{\rho,\nu,-*}_i,~1\leq i\leq D$ (rsp. $u^{\rho,\nu,-,*}_i,~1\leq i\leq D$) is similar. For some constant $c>0$
the conclusion holds if for some $1\leq i\leq 3$ we have 
\begin{equation}\label{hypo2proof}
\lim_{\nu\downarrow 0}\overline{\lim}_{\sigma\uparrow \infty}{\big |}u^{\rho,\nu,-}_i(\sigma,y^{\rho}){\big |}=\lim_{\nu \downarrow 0}\overline{\lim}_{t\uparrow \rho}{\big |}u^{\rho,\nu,-*}_i(t(\sigma),y^{\rho}){\big |}= c\neq 0
\end{equation}
for some regular function $u^{\rho}_i=\lim_{\nu\downarrow 0}u^{\rho,\nu,-}_i$ which satisfies the limit equation ($\nu\downarrow 0$) for $u^{\rho,\nu,-}_i,¸1\leq i\leq D$ considered below.
If there is some $\tilde{c}\neq 0$ and $y\in {\mathbb R}^3$ where $h_i(0)=\tilde{c}\neq 0$ for some $1\leq i\leq D$, then local time analysis of the function $u^{\rho,\nu',-,*}_i,~1\leq i\leq D$ (or related local time analysis $u^{\rho,\nu,-,*}_i,~1\leq i\leq D$) and corresponding global time analysis of the $u^{\rho,\nu,-}_i,~1\leq i\leq D$) shows that for small $\rho>0$ we have ${|}u^{\rho}_i(\rho,y){|}\geq \frac{\tilde{c}}{2}$ independently of $\nu >0$. This implies that there are singular solutions for a dense set of data as described. The details of the local analysis of the equation for $u^{\rho,\nu}_i,¸1\leq i\leq D$ and the viscosity limit are considered in the last section.
\end{proof}

\section{Vorticity singularities and local turbulence in compressible Navier Stokes equation systems}

We consider the dynamic relevance of vorticity solutions of the Euler equation as inviscid limits of solutions of the incompressible Navier Stokes equation system. The incompressible Navier Stokes equation is in turn a simplification of the physically more realistic standard compressible Navier Stokes equation systems for the dynamic variables $\rho_m$ (mass density) $\mathbf{v}=\left(v_1,\cdots ,v_D \right)^T$ (velocity field), $p$ (pressure) and $T$ (temperature)  
of the form
\begin{equation}\label{navcomp}
\left\lbrace\begin{array}{ll}
\frac{\partial \rho_m}{\partial t}+\sum_{i=1}^3\frac{\partial}{\partial x_i}\left(\rho_m v_i\right)=\frac{\partial \rho_m}{\partial t}+\sum_{i=1}^3v_i\frac{\partial}{\partial x_i}\rho_m+\rho_m \mbox{div} v=0\\
\\ 
\rho_m\frac{\partial\mathbf{v}}{\partial t}-\mu \Delta \mathbf{v}+ \rho_m(\mathbf{v} \cdot \nabla) \mathbf{v} = - \nabla p+\mu'\nabla\left(\nabla \cdot \mathbf{v}\right)+\rho_m \mathbf{f} , \\
\\
\frac{\partial (\rho_m e)}{\partial t}+\nabla\left(\left( \rho_m e+p\right) \mathbf{v} \right)=\nabla\cdot\left(\Sigma\cdot \mathbf{v}\right)+\rho_m\mathbf{f}\cdot \mathbf{v}+\nabla \cdot\left(\lambda \nabla T\right)\\
\\
\rho_m \frac{R}{M}T=p\\
\\
\mathbf{v}(0,.)=\mathbf{h},~T(0,x)=T_0(x),~p(0,x)=p_0(x),~\rho_m(0,x)=\rho_{m0}(x),
\end{array}\right.
\end{equation}
where the energy $e$ is determined by the equation  
\begin{equation}
e=u+\frac{1}{2}|\mathbf{v}|^2, \mbox{ along with }u=\frac{R}{(\gamma-1)}T,
\end{equation}
and where the latter equation determines the inner energy $u$ (according to Boyle-Mariott)
and where $\gamma$ is an adiabatic constant and $R$ is a gas constant. The additional constants or functions $M$ and $\lambda$ are exogenous as is the force $\mathbf{f}$. Natural physical constraints are $T_0>0$, $p_0>0$, $\rho_{m0}>0$ and regular (smooth) data $\mathbf{h}=(h_1,\cdots ,h_D)^T$ of polynomial decay of any order. It is common sense that turbulence is related to high Reynold numbers ${\it Re}\sim \frac{\rho}{\eta}$ at certain points of space-time. Singularities of solutions are mainly caused by degeneracies of the mass density $\rho_m$ or of the Lam\'{e}-viscosity $\mu$. Depending on the material this viscosity may be considered to be a constant or itself dependent on space and time. Similar for $\mu'$. For variable Lam\'{e} viscosity $\mu$ we define the set $D_e$ of real degeneracies to be the set of all points which have a degenerate cylinder neighborhood in $D_e$, i.e., all space-time points $(t,x)\in \left[0,\infty \right)\times {\mathbb R}^3$ which satisfy $\mu(t,y)=0$ and such that $(t_0,t_1)\times B_{\epsilon}(x)\subseteq D_e$ for some $\infty >t_1-t_0>0$ and some ball $B_{\epsilon}(x)$ of radius $\epsilon >0$ around  $x$ we have $\mu(t,y)=0$. If there is no such space time point then $D_e$ is the empty set. Otherwise we say that $D_e$ is the real set of degeneracies of (\ref{navcomp}). Turbulence related to high Reynold numbers can only be caused by small $\mu$ or high values of mass density $\rho$. For large mass density the second equation in (\ref{navcomp}), i.e., the essential equation for the dynamics, reduces to an Navier Stokes equation equation with force term which is dynamically as stable as the force term allows . As we have no mathematical theory of turbulence we may use 'weak' notions of turbulence which are strictly compatible with standard qualitative description. Next we argue that a) the singular inviscid limit vorticity of Euler equations is essentially a local weak concept of turbulence (compatible with the description of turbulence in \cite{LL}), and b) if there is real set of degeneracies $D_e\neq \oslash$ of (\ref{navcomp}) which admits a regular solution on the complementary set, then there is a solution branch of (\ref{navcomp}) with at least one vorticity singularity in each connected component of $D_e$.  
 As a qualitative description of turbulence we consider paragraph 31 of \cite{LL}. The list of notions of the concept of turbulence associated to high Reynold numbers given there consists of $\alpha)$ an extraordinary irregular and disordered change of velocity at every point of a space-time area, $\beta)$ the velocity fluctuates around a mean value, $\gamma)$ the amplitudes of the fluctuations are not small compared to the magnitude of the velocity itself in general, $\delta)$ the fluctuations of the velocity can be found at a fixed point of time, $\epsilon)$ the trajectories of fluid particles are very complicated causing a strong mixture of the fluid. A weak quantitative concept of turbulence is a list of properties of possible solutions of a Navier Stokes equation or an Euler equation such that instances of this list can be subsumed by the notions in the list $\alpha)$ to $\epsilon)$ on a real set of space time points. The proof of the theorem \ref{tm3} follows from the construction of local vorticity singularities and some additional observations of the next section.   
\begin{thm}\label{tm3}
Assume that the compressible Navier Stokes equation system in (\ref{navcomp}) has a real set of degeneracies $D_e$ in the sense defined above. Assume furthermore that the Cauchy problem for (\ref{navcomp}) has a global regular solution outside the set $D_e$ of degeneracies with strictly positive mass density. Then in each connected component of $D_e$ there is at least one singular point of vorticity, where each such singular point is a weak concept of turbulence.
\end{thm}
Singular vorticity in space-time sets of real degeneracies is closely related  to the local analysis which is needed to prove the hypothesis of theorem \ref{hyposing2} which is indeed the next step.  

\section{Analysis of the regular factors $u^{\rho,\nu,-}_i,¸1\leq i\leq D$ of the vorticity function and their viscosity limit $\nu\downarrow 0$}

Next we consider the local analysis of the regular part $u^{\rho,\nu,-}_i, 1\leq i\leq D$ of the 'vorticity component functions' $\omega^{\rho,\nu,-}_i,~1\leq i\leq D$, and the viscosity limit $\nu\downarrow 0$. Note that our considerations imply a positive answer to the question in (5.2) on p. 168 of \cite{MB}. In order to achieve this we first describe a local scheme for the functions $u^{\rho,\nu,-,*}_i,~1\leq i\leq D$ or the extended functions $u^{\rho,\nu,-,*,e}_i,~1\leq i\leq D$ corresponding to a global scheme for the functions $u^{\rho,\nu,-}_i,~1\leq i\leq D$ or the extended functions $u^{\rho,\nu,-,e}_i,~1\leq i\leq D$ and show the existence of uniformly bounded solutions of the equation in (\ref{transformed-}). We note that we do not need global existence arguments for the Navier Stokes equation or closely related equations for our argument, since a global solution of (\ref{transformed-}) is a local equation in original time coordinates, and it is not that hard to get regular upper bounds for solutions of local equations.  More precisely, we can prove global existence for the equivalent equations in $(\sigma,y^{\rho})$-coordinates or local regular existence for the equations in $(t,y^{\rho})$-coordinates with original time coordinates, where we want the original time interval to become small for $\rho$ small. In addition to the existence result we need the verification of the fact that the regular factor $u^{\rho}_i$ (the viscosity limit of $u^{\rho,\nu,-}_i)$ of the vorticity $\omega_i$ is not zero at the crucial point $(\rho,0)$ at the tip of the cone (in original time coordinates) where we want to detect the singularity. This requires a) that the solution function $u^{\rho,\nu,-,*}_i,~1\leq i\leq D$ (or the approximation $u^{\rho,\nu,-,*,e}_i,~1\leq i\leq D$) does not change too much on the time interval $[0,\rho]$ with respect to original time $t$ (independently of the viscosity  $\nu>0$), and b) that there is
a upper bound of the solution  of the equation (\ref{transformedrho}) below in transformed coordinates $(\sigma,y^{\rho})$ in a pointwise and strong sense, and which is again independent of the viscosity. We emphasize that this is essentially a local upper bound in original time coordinates. 

The requirement a) can be verified by local contraction results.
Furthermore, for part b) we do not have to rely on arguments concerning upper bounds for the incompressible Navier Stokes equation in strong norms, because $b\alpha)$ we have a correspondence to a local Navier stokes equation or to equations which are close related to the Navier Stokes equation, and $b\beta)$  the Navier Stokes type equation considered here has a potential term with 'the right sign', i.e., a term which causes a damping effect -  which naturally leads to regular solutions and global upper bounds in strong norms exploiting some spatial effects of the operator. Again we 
emphasize that for the argument of this paper we need only a local regular existence result, and we get this a fortiori for the regular factor $u^{\rho,\nu}_i,¸1\leq i\leq D$ as we have it for the original velocity function $v_i,~1\leq i\leq D$. Note that local existence in original $t$-coordinates corresponds to global existence in $\sigma $-coordinates. 
Next, we define a global scheme for the equation in $\left( \sigma,y^{\rho}\right)$-coordinates. For the convenience of the reader we write down the derivation of this equation in detail. For transformed coordinates
\begin{equation}
\left(\sigma,y^{\rho}\right)=\left(\frac{\frac{t}{\rho}}{\sqrt{1-\frac{t^2}{\rho^2}}}, y^{\rho}\right)=\left(\frac{t}{\sqrt{\rho^2-t^2}},\left( \rho-t\right) \frac{2}{\pi}\arctan(x) \right)   
\end{equation}
we define for $t\in [0,\rho)$ and $1\leq i\leq D$
\begin{equation}
\omega^{\rho,\nu,-}_i(t,x)=\frac{u^{\rho,\nu,-}_i\left(\sigma,y^{\rho} \right) }{\rho-t}.
\end{equation}
Recall that the minus upper script refers to the fact that we add some Laplacian $-\nu' \Delta u^{\rho,\nu,-}_i$ to the equation after the transformation (cf. first Section). Note that we are flexible with the choice of $\nu'$.  Especially we may we choose $\nu'$ such that the analysis simplifies slightly in original coordinates.  
When it comes to transformations of the Leray projection term, we shall have a measure change again. We note
\begin{equation}
dy^{\rho}=\Pi_{i=1}^ndy^{\rho}_i=\left( \rho-t\right)^n \frac{2^n}{\pi^n}\Pi_{i=1}^n\frac{1}{1+x_i^2}dx_i.
\end{equation}
We use the abbreviations
\begin{equation}
b^{\rho}_i=y^{\rho}_i=\left( \rho-t\right) \frac{2}{\pi}\arctan(x_i),~b^{\rho}_i=\left(\rho-t\right) b^{\rho,s}_i.
\end{equation}
For convenience we derive the equation for the function $u^{\rho,\nu}_i,~1\leq i\leq D$ explicitly. The function $u^{\rho,\nu,-}_i,~1\leq i\leq D$ is then determined by the equation which replaces the transformed Laplacian by a simple Laplacian $\nu'\Delta u^{\rho,\nu,-}_i,~ 1\leq i\leq D$.   We have 
\begin{equation}
\frac{\partial \omega^{\rho,\nu}_i}{\partial t}=\frac{u^{\rho,\nu}_i(\sigma,x)}{(\rho-t)^2}+\frac{u^{\rho,\nu}_{i,\sigma}}{\rho-t}\frac{d\sigma}{dt}+\sum_{j=1}^D\frac{u^{\rho,\nu}_{i,j}}{(\rho-t)}\frac{\partial b^{\rho}_j}{\partial t},
\end{equation}
where $\frac{\partial \sigma}{\partial t}=\frac{1}{\sqrt{\rho^2-t^2}}+t\frac{-1}{2}\sqrt{\rho^2-t^2}^{-3}(-2t)=\frac{\rho^2-t^2+t^2}{\sqrt{\rho^2-t^2}^3}=\frac{\rho^2}{\sqrt{\rho^2-t^2(\sigma)}^3}$. Note that
\begin{equation}
b^{\rho}_{j,t}=\frac{\partial y^{\rho}_j}{\partial t}=-\frac{2}{\pi}\arctan(x_j)
\end{equation}
are bounded smooth coefficients with bounded derivatives on the domain where the functions are supported. Since we have a diagonal transformation, we have
\begin{equation}
\omega^{\rho,\nu}_{i,j}=\frac{1}{\rho-t}u^{\rho,\nu}_{i,j}b^{\rho}_{j,j}.
\end{equation}
Next for the Burgers term we denote the transformed kernel by $K_3(x)=K^{*,\rho}_3(y^{\rho})$, where we may assume that the domain of $K^{*,\rho}_3$ is trivially extended to ${\mathbb R}^D$ such that $K^{*,\rho}_3(z)=0$ for $z\in {\mathbb R}^D\setminus K^{\rho}_0$. 
\begin{equation}\label{velkern}
K_3(x)h=\frac{1}{4\pi}\frac{x\times h}{|x|^3}=K^{*,\rho}_3(y^{\rho})h=\frac{1}{4\pi}\frac{x(y^{\rho})\times h}{|x(y^{\rho})|^3},
\end{equation}
where $x(y)=(x_1(y_1),\cdots,x_D(y_D))$ with
\begin{equation}
x_i(y_i)=\tan\left( \frac{\pi}{2}\frac{y^{\rho}_i}{\rho-t}\right). 
\end{equation}
Next the Burgers term
\begin{equation}
\sum_{j=1}^D\left( \int_{{\mathbb R}^3}K_3(y)\omega^{\rho,\nu}_j(t,x-y)dy\right) \frac{\partial \omega^{\rho,\nu}_i}{\partial x_j}
\end{equation}
becomes
\begin{equation}
\frac{1}{(\rho-t)^2}\sum_{j=1}^D\left( \int_{{\mathbb R}^3}K^{*,\rho}_3(w)u^{\rho,\nu}_j(\sigma,y-w)\right)\frac{\pi^3}{2^3}\frac{\Pi_{i=1}^n(1+x_i^2)}{\left (\rho-t\right)^n}dw u^{\rho,\nu}_{i,j}b^{\rho}_{j,j},
\end{equation}
and the Leray projection term
\begin{equation}
\sum_{j=1}^D\frac{1}{2}\left(\frac{\partial v^{\rho,\nu}_i}{\partial x_j}+\frac{\partial v^{\rho,\nu}_j}{\partial x_i} \right)\omega^{\rho,\nu}_j
\end{equation}
becomes
\begin{equation}
\begin{array}{ll}
\frac{1}{(\rho-t)^2}\sum_{j=1}^D\frac{1}{2}{\Bigg (}\int_{K^{\rho}_0}K^{*,\rho}_3(w)u^{\rho,\nu}_{i,j}(\sigma,y-w)b^{\rho}_{j,j}\frac{\pi^3}{2^3}\frac{\Pi_{i=1}^n(1+x_i^2)}{\left (\rho-t\right)^n}dw\\
\\
\hspace{2.5cm}+\int_{K^{\rho}_0}K_3(y)u^{\rho,\nu}_{j,i}(\sigma,y-w)b^{\rho}_{i,i}\frac{\pi^3}{2^3}\frac{\Pi_{i=1}^n(1+x_i^2)}{\left (\rho-t\right)^n}dw {\Bigg )}u^{\rho,\nu}_j.
\end{array}
\end{equation}
We do not need to compute the transform of the Laplacian as we are interested in the function $u^{\rho,\nu,-}_i,~ 1\leq i\leq D$ anyway. For this function we get the equation
\begin{equation}
\begin{array}{ll}
\frac{u^{\rho,\nu,-}_i(\sigma,x)}{(\rho-t)^2}+\frac{u^{\rho,\nu,-}_{i,\sigma}}{\rho-t}\frac{\rho^2}{\sqrt{\rho^2-t^2(\sigma)}^3}+\sum_{j=1}^3\frac{u^{\rho,\nu,-}_{i,j}}{(\rho-t)}\frac{\partial b^{\rho}_j}{\partial t}
-\nu' \Delta u^{\rho,\nu,-}_i
\\
\\
+\frac{1}{(\rho-t)^2}\sum_{j=1}^3\left( \int_{K^{\rho}_0}K^*_3(w)u^{\rho,\nu,-}_j(t,y-w)\frac{\pi^3}{2^3}\frac{\Pi_{i=1}^n(1+x_i^2)}{\left (\rho-t\right)^n}dw\right) u^{\rho,\nu,-}_{i,j}b^{\rho}_{j,j}\\
\\
=\frac{1}{(\rho-t)^2}\sum_{j=1}^3\frac{1}{2}{\Bigg (}\int_{K^{\rho}_0}K_3(w)u^{\rho,\nu,-}_{i,j}(\sigma,y-w)b^{\rho}_{j,j}\frac{\pi^3}{2^3}\frac{\Pi_{i=1}^n(1+x_i^2)}{\left (\rho-t\right)^n}dw\\
\\
\hspace{2.5cm}+\int_{K^{\rho}_0}K^*_3(w)u^{\rho,\nu,-}_{j,i}(\sigma,y-w)b^{\rho}_{i,i}\frac{\pi^3}{2^3}\frac{\Pi_{i=1}^n(1+x_i^2)}{\left (\rho-t\right)^n}dw {\Bigg )}u^{\rho,\nu,-}_j,
\end{array}
\end{equation}
where $\nu'$ may be chosen.
Abbreviating
$\overline{\mu}^{\rho}=\frac{1}{\rho^2}\sqrt{\rho^2-t(\sigma)^2}^3$, 
\begin{equation}
u^{\int,\rho,\nu,-}_{j}=
\int_{K^{\rho}_0}K^*_3(w)u^{\rho,\nu,-}_j(\sigma,y-w)\frac{\pi^3}{2^3}
\frac{\Pi_{i=1}^3(1+x_i^2}{\left(\rho-t\right)^3}dw,
\end{equation}
and
\begin{equation}
(b^{\rho,s}_{j,j} u^{\int,\rho,\nu,-}_{i,j})=
\int_{K^{\rho}_0}K^*_3(w)u^{\rho,\nu}_{i,j}(\sigma,y-w)b^{\rho,s}_{j,j}\frac{\pi^3}{2^3}
\frac{\Pi_{i=1}^3(1+x_i^2}{\left(\rho-t\right)^3}dw,
\end{equation}
and using $b^{\rho}_i=(\rho-t)b^{\rho,s}_i$ the equation takes the form 
\begin{equation}\label{transformedrho}
\begin{array}{ll}
u^{\rho,\nu,-}_{i,\sigma}-\nu' \Delta u^{\rho,\nu,-}_{i}
+\overline{\mu}^{\rho}\sum_jb^{\rho}_{j,t}
 u^{\rho,\nu,-}_{i,j}
 +\overline{\mu}^{\rho}\sum_{j}b^{\rho,s}_{j,j}u^{\int,\rho,\nu,-}_{j}u^{\rho,\nu,-}_{i,j}\\
 \\+\frac{1}{\rho^2}\frac{\sqrt{\rho^2-t^2(\sigma)}^3}{\rho-t}u^{\rho,\nu,-}_i
=\overline{\mu}^{\rho}\sum_j\frac{1}{2}\left( (b^{\rho,s}_{j,j} u^{\int,\rho,\nu,-}_{i,j})+(b^{\rho,s}_{i,i} u^{\int,\rho,\nu,-}_{j,i})\right)u^{\rho,\nu,-}_j.
\end{array}
\end{equation}

Again, natural approximations by probabilistic representations in terms of convolutions with the Gaussian related to a solution of (\ref{transformed}) are denoted by $u^{\rho,\nu,-,e}_{i},~ 1\leq i\leq D$, where these solutions on the whole domain $\left[0,\infty \right)\times {\mathbb R}^D$ tend to zero outside the cone in the viscosity limit $\nu,\nu'\downarrow 0$.
The first line of equation (\ref{transformedrho}) represents a parabolic linear scalar operator with bounded regular coefficients which depend on space and time
plus a Burgers coupling term in vorticity form with space- and time-dependent coefficients. The right side of the second line in (\ref{transformedrho})  represents a Leray projection term, and the first term in the second line of equation (\ref{transformedrho}) represents a potential term which has the right sign in order to have a damping effect. This potential term is the only term where coefficients do not become small for small $\rho>0$ and small time $t<\rho$, i.e., 
\begin{equation}\label{26}
\lim_{\rho\downarrow 0}\frac{1}{\rho^2}\frac{\sqrt{\rho^2-t^2(\sigma)}^3}{\rho-t(\sigma)}\neq 0,
\end{equation}
while for all other coefficients we have this behavior, i.e, $\lim_{\rho\downarrow 0}\overline{\mu}^{\rho}b^{\rho,s}_{k,j}=0$ etc.. In order to observe this for the Leray projection term it needs a little work, which is done below. Note that
\begin{equation}\label{26}
\lim_{t\downarrow 0}\frac{1}{\rho^2}\frac{\sqrt{\rho^2-t^2(\sigma)}^3}{\rho-t(\sigma)}=1,
\end{equation}
for all $\rho >0$ such that we have a finite potential term for all $\rho >0$.
Furthermore note the term next to the limit in (\ref{26}) can be written as $\sqrt{\rho-t}\sqrt{\rho+t}^3/\rho^2$, and, hence, has an uniform upper bound which holds for all time. More precisely, as the latter term is $\sqrt{1-(t/\rho)^2}\sqrt{1+(t/\rho)}^2$, we surely have the upper bound $2$ for all time, and as we integrate this upper bound over a small time horizon we conclude that we get a small change of the the initial data if $\rho$ becomes small. For larger transformed time $\sigma$ we have to take spatial effect into account.  
Note that in original time coordinates we have to cancel the $\frac{d\sigma}{dt}$-term, i.e.,
\begin{equation}
u^{\rho,\nu,-,*}_{i,t}=u^{\rho,\nu,-}_{i,\sigma}\frac{d\sigma}{dt}=u^{\rho,\nu,-}_{i,\sigma}\frac{\rho^2}{\sqrt{\rho^2-t^2}^3},~\mbox{or}~u^{\rho,\nu,-}_{i,\sigma}=\overline{\mu}^{\rho}u^{\rho,\nu,-,*}_{i,t},
\end{equation}
such that we get from (\ref{transformedrho}) and with $\nu^*=\frac{1}{\overline{\mu}^{\rho}}\nu$ the equation
\begin{equation}\label{transformed2*erst}
\begin{array}{ll}
u^{\rho,\nu,-,*}_{i,t}-\nu^* \Delta u^{\rho,\nu,-,*}_{i}
+\sum_jb^{\rho}_{j,t}
 u^{\rho,\nu,,*}_{i,j}\\
 \\
 +\sum_{j}b^{\rho,s}_{j,j}u^{\int,\rho,\nu,*}_{j}u^{\rho,\nu,-,*}_{i,j}+\frac{1}{\rho-t}u^{\rho,\nu,-,*}_i\\
 \\
=\sum_j\frac{1}{2}\left( (b^{\rho,s}_{j,j} u^{\int,\rho,\nu,-,*}_{i,j})+(b^{\rho,s}_{i,i} u^{\int,\rho,\nu,-,*}_{j,i})\right)u^{\rho,\nu,-,*}_j,
\end{array}
\end{equation}
where we use some abbreviations as before. Again we may consider the equation in (\ref{transformed2*erst}) on the whole domain (with trivially extended coeffcients), and we denote a solution of the extended equation by 
\begin{equation}
u^{\rho,\nu,-,*,e}_{i},~1\leq i\leq D.
\end{equation}
Note that the viscosity limit is then supported on the corresponding cone.
We observe that we have a strong damping term which makes local regular upper bounds more than plausible. However it is not too strong in the sense, that this does not imply that the value function in original time coordinates  becomes zero at $(\rho,0)$, because the function has its main mass (with respect to any usual norm) on the cone $K^{\rho,t}_0=[0,\rho)\times \cup_{0\leq t\leq 1}\left(-\rho+t,\rho-t\right) ^3$. More precisely, solutions  $u^{\rho,\nu,-,*,e}_{i},~1\leq i\leq D$ have an exponential spatial decay outside the cone $K^{\rho}_0$ which disappears in the viscosity limit. Such an representation in terms of the fundamental solution $G^*_{\nu}$ of the equation
\begin{equation}
G^*_{\nu,t}-\nu^*  \Delta G^*_{\nu}=0,
\end{equation}
has the form 
\begin{equation}\label{transformed2*}
\begin{array}{ll}
u^{\rho,\nu,-,*,e}_{i}(t,y)=\left( u^{\rho,\nu,-,*,e}_{i}(0,.)\ast_{sp}G^*_{\nu}\right)(t,y)\\
\\
-\left( \sum_jb^{\rho}_{j,t}u^{\rho,\nu,,*}_{i,j}\ast G^*_{\nu}\right)(t,y)\\
\\
-\int_{D_{\rho}}\left( \sum_{j}b^{\rho,s}_{j,j}u^{\int,\rho,\nu,-,*,e}_{j}u^{\rho,\nu,-,*,e}_{i,j}\right) (s,z) G^*_{\nu}(t,y;\tau,z)dzd\tau\\
\\
-\int_{D_{\rho}}\frac{1}{\rho-\tau}u^{\rho,\nu,-,*,e}_i(s,z)G^*_{\nu}(t,y;\tau,z)dzd\tau\\
 \\
+\int_{D_{\rho}}\left( \sum_j\frac{1}{2}\left( b^{\rho,s}_{j,j} u^{\int,\rho,\nu,*,e}_{i,j}+b^{\rho,s}_{i,i} u^{\int,\rho,\nu,-,*,e}_{j,i}\right)u^{\rho,\nu,-,*,e}_j\right) (\tau,z)G^*_{\nu}(t,y;\tau,z)dzd\tau.
\end{array}
\end{equation}
The fact that for small $\nu^* >0$ the main mass of the functions $u^{\rho,\nu,-,*,e}_{i},~ 1\leq i\leq D$ is on a cone leads to appropriate upper and lower bounds of the damping term (as we shall observe below the factor $\frac{1}{\rho-\tau})$ is cancelled  in the viscosity limit by the measure change from a cone to a cylinder) .
In any of the two variations of argument for local time $t$ and all time $\sigma$ the damping term dominates potential growth terms if data become large in a time-discretized scheme. Furthermore, as the mass of the value functions $u^{\rho,\nu,-,e}_i,~1\leq i\leq D$ is mainly on the cone $K^{\rho}_0$ we shall observe that the damping effect does not only dominate any growth which may due to other terms, but has also an upper bound which becomes small compared to nonzero initial data $u^{\rho,\nu,-,e}_{i}(0,0)=c\neq 0$ in order to obtain that $u^{\rho,\nu,-,*,e}_i(\rho,0)\neq 0$ and $u^{\rho,-,*,e}_i(\rho,0)=u^{\rho,-,*}_i(\rho,0)\neq 0$ for some $i$  in the viscosity limit $\nu^*_k \downarrow 0$ for a suitable subsequence $(\nu^*_k)$) such that we detect a singularity. 
The program of proof is then as follows:
\begin{itemize}
 \item[i)] for $\rho>0$ small we obtain contraction results in strong Banach spaces and a series of linear approximating time-local solutions converging to a solution of the nonlinear equation for $u^{\rho,\nu,-,*,e}_i,~1\leq i\leq 3$ on a local time interval corresponding to a solution $u^{\rho,\nu,-,e}_i,~1\leq i\leq 3$ on a global time interval. We observe that for a class of regular data local contraction holds with a contraction constant which is independent of the viscosity $\nu>0$.
\begin{rem} 
Independence of the contraction constant of the viscosity $\nu$ for regular data is essential. In the last step of the proof we use compactness arguments in order to show that there is a sequence $(\nu_k)_{k\geq 1}$ and a viscosity limit such that  $\lim_{\nu_k\downarrow 0}\omega^{\rho,\nu_k,-}_i=(\rho-t)\lim_{\nu_k\downarrow 0}u^{\rho,\nu_k,-,*}_i,~1\leq i\leq 3$, on a time interval $[0,\rho)$, i.e., $\lim_{\nu_k\downarrow 0}u^{\rho,\nu_k,-,*}_i,~1\leq i\leq 3$ is the regular part of a local time vorticity solution. Furthermore, we remark that local contraction leads to local time existence for the   function $u^{\rho,\nu,-,*,e}_i,~1\leq i\leq 3$ and global time existence for $u^{\rho,\nu,-,e}_i,~1\leq i\leq 3$. It is convenient to use these approximation functions and construct a viscosity limit $\nu^*_k\downarrow 0$ for these functions.
 \item[ii)] the contraction result and local existence result in time $t$-coordinates of item i), corresponding to a global existence result in time $\sigma$-coordinates, leads to a global regular upper bound for the value function $u^{\rho,\nu,-,e}_i,~1\leq i\leq 3$ in a regular norm for regular data. Similarly for the value function $u^{\rho,\nu,-,*,e}_i,~1\leq i\leq 3$. 
\end{rem}
\item[iii)]
The regular upper bound of item ii) are independent of the viscosity parameter $\nu$. Here we use a) probabilistic representations of local solutions, b) Lipschitz continuity of the convoluted data related to the nonlinear Burgers term and the Leray projection term, and c) preservation of a certain order of polynomial dacay at spatial infinity, if the data have a certain order of polynomial decay. The latter observation allows us to obtain stronger compactness arguments compared to the classical Rellich embedding theorems.  

Upper bounds for the modulus of the viscosity limit of the value function $u^{\rho,\nu,-,*,e}_i,~1\leq i\leq 3$  are obtained from classical representations with terms of the form
 \begin{equation}
 g\ast G_{\nu},~g\ast_{sp}G_{\nu}
 \end{equation}
with the Gaussian $G_{\nu}$. For the viscosity limit we use (with $y=(y_1,\cdots,y_n)$ , $y_i=\sqrt{\nu}z_i$) that for $t> 0$ we have 
\begin{equation}
\begin{array}{ll}
\int_0^t\int{{\mathbb R}^D}\frac{1}{\sqrt{\sqrt{4\pi \nu s}^n}}\exp\left(-\frac{|y|^2}{4\nu s} \right)dy ds=\\
\\
\int_0^t\int{{\mathbb R}^D}\frac{1}{\sqrt{\sqrt{4\pi  s}^n}}\exp\left(-\frac{|z|^2}{4 s} \right)dz ds=\int_0^t ds=t
\end{array}
\end{equation}
which is independent of $\nu >0$. Here, $g\ast G_{\nu},~g\ast_{sp}G_{\nu}$ denotes the convolution and spatial convolution with the Gaussian respectively. For the construction of the upper bounds for spatial derivatives of the function $u^{\rho,\nu,-,*,e}_i,~1\leq i\leq 3$ we need regularity of the data as classical representations include terms of the form
\begin{equation}
 h\ast G_{\nu,i}\mbox{ next to }g\ast G_{\nu},~g\ast_{sp}G_{\nu}
 \end{equation}
 for some regular functions $h$ (this $h$ without index is not to be confused with that data function of the problem with index). We can then use regularity of this $h$ (available for some regular data by the iteration scheme) and use the symmetry of the first order spatial derivative of the Gaussian in estimates of the form
\begin{equation}
\begin{array}{ll}
 {\Big |}\int_{{\mathbb R}^D}h(y)\frac{2(x_i-y_i)}{4\nu s\sqrt{4\pi \nu s}^n}\exp\left(-\frac{|x-y|^2}{4\nu s}\right) dy{\Big |}\\
 \\
 \leq \int_{\left\lbrace y|y\in {\mathbb R}^D, y_i\geq 0\right\rbrace }{\Big |}h(x-y)-h(x-y^{i,-}){\Big |}\frac{2|y_i|}{4\nu s\sqrt{\pi \nu s}^n}\exp\left(-\frac{|y|^2}{4\nu s}\right) dy\\
 \\
 = \int_{\left\lbrace y|y\in {\mathbb R}^D, y_i\geq 0\right\rbrace }{\Big |}h(y)-h(y^{i,-}){\Big |}\frac{2|x_i-y_i|}{4\nu s\sqrt{\pi \nu s}^n}\exp\left(-\frac{|x-y|^2}{4\nu s}\right) dy.
\end{array}
 \end{equation}
Here, $y^{i,-}=\left(y^{i,-}_1,\cdots,y^{i,-}_D\right)$, where $y^{i,-}_j=y_j$ for $j\neq i$ and $y^{i,-}_i=-y_i$.
Then Lipschitz regularity of functions $h$ can be used in order to obtain an upper bound in viscosity limit, and transfer local regular existence in original time etc.. 
 \item[iv)] Finally, we construct viscosity limits and choose approriate data such that a) the viscosity limit exists in a classical regular space, and b) choose the data such that the appropriate limits exist and such that the functional increment $\delta u^{\rho,\nu,-,*}_i(\rho,0)=u^{\rho,\nu,-,*}_i(\rho,0)-u^{0,\nu,-,*}_i(0,0)\neq 0$ for some $1\leq i\leq 3$ and for $\nu >0$ small enough.  We choose a rather restricted set of data $S$. Larger sets can be chosen, but the stronger set allows us to consider some rough upper bound estimates. The set of regular data $h_i,~1\leq i\leq D$ may be assumed to be the set of bounded regular functions $g_i,~1\leq i\leq D$ with bounded derivatives multiplied by a factor of polynomial decay of order $13$, i.e.
\begin{equation}
h_i=\frac{g_i}{1+|x|^{13}},
\end{equation}
where $g_i$ are smooth functions with bounded derivatives, i.e., for all multiindices $\alpha$ there are finite constants $C_{\alpha}$ such that
\begin{equation}
{\big |}D^{\alpha}_xg_i{\big |}\leq C_{\alpha}<\infty.
\end{equation}
A set of functions $g_i$ which satisfies the latter criteria is the set of finite Fourier series in dimension $D$.  For $\nu>0$ small enough, and for some constants $C,c>0$ we have  $|h_i|_{H^m\cap C^m}\leq C$ for $m\geq 2$ and if $\max_{1\leq i\leq 3}u^{\rho,\nu,-,*,e}_i(0,0)=h_{i_0}(0)=c\neq 0$ for some index $i_0$ then there exists some small $\rho >0$ such that the maximum increment has an upper bound $\max_{1\leq i\leq 3}{\big |}\delta u^{\rho,\nu,-,*,e}_i(\rho,0){\big |}\leq \frac{1}{2}c$, which holds also in the viscosity limit $\nu\downarrow 0$. Here special attention is needed concerning the Leray projection term as there is a measure change of the Laplacian kernel if this term is considered in transformed coordinates. Hence $\nu >0$ small enough the regular part $u^{\rho,-}_i,~1\leq i\leq D$ of the 'vorticity' $\omega^{\nu,-}_i$ at the crucial point $(\rho,0)$ is not equal to zero. Here the increment is $\delta u^{\rho,\nu,-,*}_i(\rho,0)=u^{\rho,\nu,-,*}_i(\rho,0)- u^{\rho,\nu,-,*}_i(0,0)$. Note that in the latter statement we ommit the upper scripts $^*$ and $^e$. This leads to the conclusion of singular solutions for the Euler equation vorticity itself.  
\end{itemize}
Next we consider each item in detail. Ad i), for an iteration index $q\geq 1$  the functions $u^{\rho,\nu,-,q}_{i},~1\leq i\leq D,~q\geq 1$ are determined by $u^{\rho,\nu,-,0}_i=f_i(0,.)$ for $q= 1$ and, recursively, as solutions of linearized equations of the form
\begin{equation}\label{transformediter}
\begin{array}{ll}
u^{\rho,\nu,-,e,q+1}_{i,\sigma}-\nu^* \Delta u^{\rho,\nu,-,e,q+1}_{i}\\
\\
+\overline{\mu}^{\rho}\sum_jb^{\rho}_{j,t}
 u^{\rho,\nu,-,e,q+1}_{i,j}
 +\overline{\mu}^{\rho}\sum_{j}b^{\rho,s}_{j,j}u^{\int,\rho,\nu,-,e,q}_{j}u^{\rho,\nu,-,e,q}_{i,j}\\
 \\
 +\frac{1}{\rho^2}\frac{\sqrt{\rho^2-t^2(\sigma)}^3}{\rho-t}u^{\rho,\nu,-,e,q+1}_i\\
 \\
=\overline{\mu}^{\rho}\sum_j\frac{1}{2}\left( (b^{\rho,s}_{j,j} u^{\int,\rho,\nu,-,e,q}_{i,j})+(b^{\rho,s}_{i,i} u^{\int,\rho,\nu,-,e,q}_{j,i})\right)u^{\rho,\nu,-,e,q}_j,
\end{array}
\end{equation}
where we applied the abbreviations introduced above. 
\begin{rem}\label{viscext}
Note that the equation does not correspond to a Navier Stokes equation as we have a simplified viscosity term. Adding the Laplacian transformation
$$\nu \overline{\mu}^{\rho} \sum_{j} b^{\rho}_{j,j}b^{\rho}_{j,j}u^{\rho,\nu,-,q+1}_{i,j,j}-\nu \overline{\mu}^{\rho}\sum_{j} b^{\rho}_{j,j,j}u^{\rho,\nu,-,q+1}_{i,j}$$ instead of a simpliefied Laplacian leads to a local iteration scheme for the Navier Stokes equation. However, this leads to more involved Gaussians, where some properties of interest are only valid in the viscosity limit. As our main interest here is in the Euler equation we consider an extension with a simplified viscosity term. This is the meaning of the $^-$-upper script.
\end{rem}
Recall also that the upper script $^e$ in (\ref{transformediter}) refers to the extended problem on the whole domain (which is a trivial extension of solution functions in the viscosity limit). 
Note that we take information from the previous iteration step only for the nonlinear terms. The reason is that the damping term with iteration index $q+1$ can be used in order to improve the contraction property (make the estimate of the contraction constant smaller). For analytic purposes this does not matter of course. We remark that a similar local analysis of the function
\begin{equation}
(t,y)\rightarrow u^{\rho,\nu,-,*,e}_i(t,y),~t\in [0,\rho],
\end{equation}
ensures local existence on the time interval $[0,\rho]$ (with support of solution function components on the cone $K^{\rho,0}_0$ for the viscosity limit $\nu\downarrow 0$) and corresponding global existence for the transformed solution function $u^{\rho,\nu,-,e}_i,~1\leq i\leq D$ exists as a consequence.
In the case of the transformed equation (defined globally in time) we can use the semigroup property, where the damping term ensures that boundedness of local solutions is inherited. Anyway, it suffices to consider the iteration scheme locally on the time intervals $[l-1,l]$ with time step size $1=l-(l-1)$.
Note that the damping term has iteration index $q+1$ as this term has the right sign and serves for a smaller contraction constant in general.  Recall that the component functions $u^{\rho,\nu,-,e}_i,~ 1\leq i\leq D$ or defined on the whole domain of ${\mathbb R}^3$ (and, in the viscosity limit become equal to zero outside the cone). Similarly, we understand that the time-local functions $u^{\rho,\nu,-,e,l}_i$ their iterative approximations $u^{\rho,\nu,-,e,l,q+1}_i$ are defined on  the whole domain of $[l-1,l]\times {\mathbb R}^3$. At each time step having computed $u^{\rho,\nu,-,e,l-1,0}_i(l-1,.):=u^{\rho,\nu,-,e,l-1}_i(l-1,.)$ at the beginning of time step $l$ we compute a series of functions $u^{\rho,\nu,-,e,l,q}_i,~1\leq i\leq 3,q\geq 0$ on the domain $[l-1,l]\times {\mathbb R}^3$ iteratively, which satisfy the Cauchy problems
\begin{equation}\label{transformediter}
\left\lbrace \begin{array}{ll}
 u^{\rho,\nu,-,e,l,q+1}_{i,\sigma}-\nu^* \Delta u^{\rho,\nu,-,l,e,q+1}_{i}\\
\\
+\overline{\mu}^{\rho}\sum_jb^{\rho}_{j,t}
 u^{\rho,\nu,-,l,e,q+1}_{i,j}
 +\overline{\mu}^{\rho}\sum_{j}b^{\rho,s}_{j}u^{\int,\rho,\nu,-,e,l,q}_{j}u^{\rho,\nu,-,e,l,q}_{i,j}\\
 \\
 +\frac{1}{\rho^2}\frac{\sqrt{\rho^2-t^2(\sigma)}^3}{\rho-t}u^{\rho,\nu,-,e,l,q+1}_i\\
 \\
=\overline{\mu}^{\rho}\sum_j\frac{1}{2}\left( (b^{\rho,s}_{j,j} u^{\int,\rho,\nu,-,e,l,q}_{i,j})+(b^{\rho,s}_{i,i} u^{\int,\rho,\nu,-,e,l,q}_{j,i})\right)u^{\rho,\nu,-,e,l,q}_j,\\
\\
u^{\rho,\nu,-,e,l,q+1}_i(l-1,.)=u^{\rho,\nu,-,e,l-1}_i(l-1,.).
\end{array}\right.
\end{equation}
In order to get local regular existence of solutions
\begin{equation}
u^{\rho,\nu,e,-,l}_i=u^{\rho,\nu,-,e,l-1}_i(l-1,.)+\sum_{m=1}^{\infty}\delta u^{\rho,\nu,-,e,l,m}_i,
\end{equation}
for $1\leq i\leq D$ and on the domain $[l-1,l]\times {\mathbb R}^3$, where
\begin{equation}
\delta u^{\rho,\nu,-,e,l,m}_i=u^{\rho,\nu,-,e,l,m}_i-u^{\rho,\nu,-,e,l,m-1}_i
\end{equation}
we can use the following local contraction result. 
\begin{thm}\label{contrthm}
Given $l\geq 1$ and $u^{\rho,\nu,-,e,l-1}_i(l-1,.)\in H^m\cap C^m$ for some $m\geq 2$ assume that for some finite constant $C>0$ 
\begin{equation}
{\big |}u^{\rho,\nu,-,e,l-1}_i(l-1,.){\big |}_{H^m}:=\sum_{|\alpha|\leq m}{\big |}D^{\alpha}u^{\rho,\nu,-,e,l-1}(l-1,.){\big |}_{L^2}\leq C.
\end{equation}
 Then there is a $\rho>0$ such that
\begin{equation}
\sup_{\sigma\in [l-1,l]}{\big |}\delta u^{\rho,\nu,-,l,e,0}_i(\sigma,.){\big |}_{H^m}\leq \frac{1}{4},
\end{equation}
and such that for all $q\geq 1$
\begin{equation}
\sup_{\sigma\in [l-1,l]}{\big |}\delta u^{\rho,\nu,-,e,l,q}_i(\sigma,.){\big |}_{H^m}\leq \frac{1}{4}\sup_{\sigma\in [l-1,l]}{\big |}\delta u^{\rho,\nu,-,e,l,q-1}_i(\sigma,.){\big |}_{H^m}.
\end{equation}
A similar result holds for the function $u^{\rho,\nu,-,*,e}_i,~1\leq i\leq D$ on the time interval $[0,\rho]$.
In case of dimension $D=3$ it is essential to have a contraction result for the $H^2$ function space, because this space includes the H\"{o}lder continuous functions such that standard regularity theory leads to the existence of classical local solutions. 
\end{thm}
\begin{proof}
We did similar estimates elsewhere, and do not repeat here all details, but sketch the reasoning and emphasize some features which are special for the equations for $u^{\rho,l,\nu,-,e}_i,~1\leq i\leq 3$. 
Given time step $l\geq 1$ for each iteration step $q$ the functional increments $\delta u^{\rho,\nu,-,e,l,q}_i,~1\leq i\leq 3$ have zero initial data at time $\sigma=l-1$, and satisfy the equation 
\begin{equation}\label{transformediter}
\begin{array}{ll}
\delta u^{\rho,\nu,-,e,l,q+1}_{i,\sigma}-\nu \Delta\delta u^{\rho,\nu,-,e,l,q+1}_{i}\\
\\
+\overline{\mu}^{\rho}\sum_jb^{\rho}_{j,t}
 \delta u^{\rho,\nu,-,e,l,q+1}_{i,j}
 +\overline{\mu}^{\rho}\sum_{j}b^{\rho,s}_{j,j}u^{\int,\rho,\nu,-,e,l,q}_{j}\delta u^{\rho,\nu,-,e,l,q}_{i,j}\\
 \\
 +\overline{\mu}^{\rho}\sum_{j}b^{\rho,s}_{j}\delta u^{\int,\rho,\nu,-,e,l,q}_{j} u^{\rho,\nu,-,e,l,q}_{i,j} +\frac{1}{\rho^2}\frac{\sqrt{\rho^2-t^2(\sigma)}^3}{\rho-t(\sigma)}\delta u^{\rho,\nu,-,e,l,q+1}_i\\
 \\
=\overline{\mu}^{\rho}\sum_j\frac{1}{2}\sum_j\left( (b^{\rho,s}_{j,j} u^{\int,\rho,\nu,-,e,l,q}_{i,j})+(b^{\rho,s}_{i,i} u^{\int,\rho,\nu,-,e,l,q}_{j,i})\right)\delta u^{\rho,\nu,-,e,l,q}_j,\\
\\
+\overline{\mu}^{\rho}\sum_j\frac{1}{2}\left( (b^{\rho,s}_{j,j} \delta u^{\int,\rho,\nu,-,e,l,q}_{i,j})+(b^{\rho,s}_{i,i} \delta u^{\int,\rho,\nu,-,e,l,q}_{j,i})\right) u^{\rho,\nu,-,e,l,q}_j.
\end{array}
\end{equation}
There are classical representations of the solutions in terms of the fundamental solution $p^{\rho,l}$ of the equation  
\begin{equation}\label{transformeddens}
\begin{array}{ll}
p^{\rho,l}_{i,\sigma}-\nu \Delta p^{\rho,l}_{i}=0.
\end{array}
\end{equation}
\begin{rem}
 Furthermore, note that there are also (approximative) classical representations with respect to the fundamental solution $\tilde{p}$ of
\begin{equation}
\tilde{p}_{,\sigma}-\nu \Delta \tilde{p}
+\overline{\mu}^{\rho}\sum_jb^{\rho}_{j,t}
 \tilde{p}_{,j}
 +\frac{1}{\rho^2}\frac{\sqrt{\rho^2-t^2(\sigma)}^3}{\rho-t(\sigma)} \tilde{p}=0,
\end{equation}
of course. However we have global regular existence theory for the linear equation, and this leads to a solution pair $u^{\rho,\nu,l,q+1}_j,u^{\rho,\nu,l,q}_j,~1\leq j\leq D$ at each iteration step $q\geq 0$, and in terms of this known pair we have at each iteration step a classical representation in terms of (\ref{transformeddens}). The solution of the linearized equation can also be achieved by a similar subiteration scheme and local contraction.
\end{rem}

The fundamental solution of this equation is indeed independent of the iteration index $l$.
Note that the density in spatial variable $y$ is supported on a cone and assumed to be trivially extended outside, and that in original spatial coordinates $x$ it is supported on the whole domain. In these original spatial variables the density has only time dependent cooefficients. Using this observation it is straightforward to observe that the density and its first order spatial derivatives satisfies classical local a priori estimates of the form  
\begin{equation}
{\big |}p^{\rho,l}(\sigma,x;s,y){\big |}\leq \frac{C}{(\sigma-s)^{\alpha}|x-y|^{n-2\alpha}}
\end{equation}
for some $\alpha\in (0,1)$, and
\begin{equation}
{\big |}p^{\rho,l}_{,j}(\sigma,x;s,y){\big |}\leq \frac{C}{(\sigma-s)^{\alpha}|x-y|^{n+1-2\alpha}}
\end{equation}
for some $\alpha\in \left(\frac{1}{2},1\right)$ such that we have local $L^1$ integrability.  
Since all coefficients of the spatial terms are bounded functions times a factor $\overline{\mu}$, they become small for a small time horizon $\rho >0$ ( which is the time horizon of the related problem in original time coordinates $t$), and, hence, the  effect of the density becomes small. The damping term can be put on the 'left side', and the increments on the right side can be extracted by use of Young inequalities. Here integrals are splitted in local intergals in a ball around weak singularities and their complements. Note that in dimension $D=3$ the kernel in the definition of the Biot-Savart law is locally $L^1$ and is $L^2$ outside a ball containing the weak singularity. The reader may consult related papers where this has been carried out in similar situations
\end{proof}

Ad ii) The existence of upper bounds is closely related to the second claim in item i) and follows quite straightforwardly from the contraction result, where the damping effect of the potential term can be used in oorder to design a global scheme for the transformed equation in global time. This global scheme corresponds to a local scheme for $u^{\rho,\nu,-,*,e,l}_i,~ 1\leq i\leq D$ on a local time interval $[0,\rho)$ in time coordinates $t$ of course. The global scheme is definied in terms of transformed local solutions of the Navier Stokes type equations for $u^{\rho,\nu,-,e,l}_i,~1\leq i\leq D$, and we describe the argument in this case. The treatment for the function $u^{\rho,\nu,-,*,e}_i,~1\leq i\leq D$ is similar.  Let us assume that we have regular data $f_i\in H^m\cap C^m$ for some $m\geq 2$. Then by induction on the time step number $l$ and by the semi-group property it for the first statement in item ii) it suffices to show the preservation of regularupper bounds of the time interval $[l-1,l]$ inductively. For each $l\geq 1$ we construct the local time limit function
\begin{equation}
u^{\rho,\nu,-,e,l}_i=u^{\rho,\nu,-,e,l,0}_i+\sum_{q=1}^{\infty}\delta u^{\rho,\nu,-,e,l,q}_i,
\end{equation}
-which is just the restriction of the function $u^{\rho,\nu,-,e}_i, 1\leq i\leq D$ to the time interval $[l-1,l]$ with respect to time $\sigma$-coordinates.
This solution function satisfies the Cauchy problem

\begin{equation}\label{transformedplug}
\left\lbrace \begin{array}{ll}
u^{\rho,\nu,-,e,l}_{i,\sigma}-\nu \Delta u^{\rho,\nu,-,e,l}_{i}
+\overline{\mu}^{\rho}\sum_jb^{\rho}_{j,t}
 u^{\rho,\nu,-,e,l}_{i,j}\\
 \\
 +\overline{\mu}^{\rho}\sum_{j}b^{\rho,s}_{j,j}u^{\int,\rho,\nu,-,e,l}_{j}u^{\rho,\nu,-,e,l}_{i,j}+\frac{1}{\rho^2}\frac{\sqrt{\rho^2-t^2(\sigma)}^3}{\rho-t}u^{\rho,\nu,-,e,l}_i\\
 \\
=\overline{\mu}^{\rho}\sum_j\frac{1}{2}\left( (b^{\rho,s}_{j,j} u^{\int,\rho,\nu,-,e,l}_{i,j})+(b^{\rho,s}_{i,i} u^{\int,\rho,\nu,-,e,l}_{j,i})\right)u^{\rho,\nu,-,e,l}_j,\\
\\
u^{\rho,\nu,-,e,l}_i(l-1,.)=u^{\rho,\nu,-,e,l-1}_i(l-1,.).
\end{array}\right.
\end{equation}
Define an 'approximation of stage $q$' by
\begin{equation}\label{ulimit}
u^{\rho,\nu,-,e,l,q}_i=u^{\rho,\nu,-,e,l,0}_i+\sum_{m=1}^{q}\delta u^{\rho,\nu,-,e,l,m}_i.
\end{equation}
Plugging in the equation in (\ref{transformedplug}), we get
\begin{equation}\label{transformedplugl}
\begin{array}{ll}
u^{\rho,\nu,-,e,l,q}_{i,\sigma}-\nu  \Delta u^{\rho,\nu,-,e,l,q}_{i}\\
\\
+\overline{\mu}^{\rho}\sum_jb^{\rho}_{j,t}
 u^{\rho,\nu,-,e,l,q}_{i,j}
 +\overline{\mu}^{\rho}\sum_{j}b^{\rho,s}_{j,j}u^{\int,\rho,\nu,-,e,l,q}_{j}u^{\rho,\nu,-,e,l,q}_{i,j}+\frac{1}{\rho^2}\frac{\sqrt{\rho^2-t^2(\sigma)}^3}{\rho-t(\sigma)}u^{\rho,\nu,-,e,l,q}_i\\
 \\
+\overline{\mu}^{\rho}\sum_j\frac{1}{2}\sum_k\left( (b^{\rho,s}_{j,j} u^{\int,\rho,\nu,-,e,l,q}_{i,j})+(b^{\rho,s}_{i,i} u^{\int,\rho,\nu,-,e,l,q}_{j,i})\right)u^{\rho,\nu,-,e,l,q}_j\\
\\
=\overline{\mu}^{\rho}\sum_{j}b^{\rho,s}_{j,j}\delta u^{\int,\rho,\nu,-,e,l,q}_{j}u^{\rho,\nu,-,e,l,q}_{i,j}\\
\\
+\overline{\mu}^{\rho}\sum_j\frac{1}{2}\left( (b^{\rho,s}_{j,j} \delta u^{\int,\rho,\nu,-,e,l,q}_{i,j})+(b^{\rho,s}_{i,i} \delta u^{\int,\rho,\nu,-,e,l,q}_{j,j})\right)u^{\rho,\nu,-,e,l,q}_j.
\end{array}
\end{equation}
Now, inductively for data $u^{\rho,\nu,-,e,l}_j(l-1,.)=u^{\rho,\nu,-,e,l-1}_j(l-1,.)\in H^m\cap C^m$ for $m\geq 2$ there is some $\rho>0$ (independent of $m,l$) such that 
\begin{equation}\sup_{\sigma\in [l-1,l]}{\big |}\delta u^{\rho,\nu,-,e,l,q}_i(\sigma,.){\big |}_{H^m\cap C^m}\downarrow 0,\end{equation} and this implies that for such a $\rho$ we have \begin{equation}\sup_{\sigma\in [l-1,l]}{\big |}\delta u^{\int,\rho,\nu,-,e,l,q}_i(\sigma,.){\big |}_{H^m\cap C^m}\downarrow 0.\end{equation} Hence, the limit $q\uparrow \infty$ of (\ref{ulimit}) satisfies the Cauchy problem in (\ref{transformedplug}) pointwise and is, hence, a classical solution.
Next it is the potential term which has the right sign in order to ensure the existence of a uniform global upper bound of the solution of (\ref{transformed}) (it has a similar role as the control function in a controlled scheme which we considered elsewhere).
Now as we have a local classical solution of the time local Navier Stokes type equation with bounded regular functions $u^{\int,\rho,\nu,-,e,l}_{j},~1\leq j\leq 3$, we may consider the fundamental solution $p^{\rho,l,B}$ of the equation
\begin{equation}\label{fund*}
\begin{array}{ll}
p^{\rho,l,B}_{i,\sigma}-\nu \Delta p^{\rho,l,B}_{i,j,j}\\
\\
+\overline{\mu}^{\rho}\sum_jb^{\rho}_{j,t}
 p^{\rho,l,B}_{i,j}+\overline{\mu}^{\rho}\sum_{j}b^{\rho,s}_{j,j} u^{\int,\rho,\nu,-,e,l}_{j}p^{\rho,l,B}_{i,j}=0,
\end{array}
\end{equation}
where we include the coefficientof the Burgers term (we added the superscript $B$ to indicate this). We can do this because we know the local solution $u^{\rho,\nu,-,e,l}_i,~1\leq i\leq 3$ by the local contraction result, and therefore we know $u^{\int,\rho,\nu,-,e,l}_i,~1\leq i\leq 3$. Note that for each time step $l\geq 1$ the fundamental solution $p^{\rho,l,B}$ is considered on the extended  domain $D_l\times D_l$ with  $(\sigma,z),(r,y)\in D_l=[l-1,l]\times {\mathbb R}^3$ and $\sigma>r$.  We emphasize that $u^{\int,\rho,\nu,-,e,l}_{j}$ is known from local existence at time step $l$ of the scheme, such that we may determine a time-local fundamental solution in order get a useful representation of $u^{\rho,\nu,-,e, l}_i,¸1\leq i\leq 3$ itself. For $\sigma\in [l-1,l]$ we have
\begin{equation}\label{transformedend}
\begin{array}{ll}
u^{\rho,\nu,-,e,l}_{i}(\sigma,x)=-
\int_{l-1}^{\sigma}\int_{{\mathbb R}^3}\frac{1}{\rho^2}\frac{\sqrt{\rho^2-t^2(r)}^3}{\rho-t(r)}u^{\rho,\nu,-,e,l}_i(r,y)p^{\rho,l,B}(\sigma,x;r,y)Dydr\\
\\
+\int_{l-1}^{\sigma}\int_{{\mathbb R}^3}\overline{\mu}^{\rho}\left( \frac{1}{2}\sum_{j}\left( b^{\rho,s}_{j,j} u^{\int,\rho,\nu,-,e,l}_{i,j}+b^{\rho,s}_{i,i} u^{\int,\rho,\nu,-,e,l}_{j,i}\right)u^{\rho,\nu,-,e,l}_j\right) (r,y)\times \\
\\
\times p^{\rho,l,B}(\sigma,x;r,y)dydr.
\end{array}
\end{equation}
 Recall that $\rho>0$ is not a step size in this article but the time horizon of the original domain, and observe that the factor 
\begin{equation}
\overline{\mu}^{\rho}=\frac{1}{\rho^2}\sqrt{\rho^2-t(\sigma)}^3\rightarrow \rho \mbox{ as } t\downarrow 0
\end{equation}
is smaller than the corresponding time factor of the damping term. Furthermore, all coefficients of the density equation in (\ref{fund*}) become small as $\rho >0$ becomes small, such that for a small time horizon of the original time problem we have a density $p^{\rho,l}$ which is close to the identity.
Hence, assuming inductively for some point $x$ that $|u^{\rho,\nu,-,e,l-1}_i(l-1,x){\big |}\in \left[ \frac{C}{2}C\right]$ we conclude from (\ref{transformedend}) that we have $|u^{\rho,\nu,-,e,l-1}_i(l-1,x){\big |} \geq |u^{\rho,\nu,-,e,l}_i(l,z){\big |}$. Hence, we get
\begin{equation}
\sup_{z\in {\mathbb R}^3}{\big |}u^{\rho,\nu,-,e,l-1}_{i}(l-1,z){\big |}\leq C
\Rightarrow \sup_{z\in {\mathbb R}^3}{\big |}u^{\rho,\nu,-,e,l}_{i}(l,z){\big |}\leq C.
\end{equation}
Inductively we get a global upper bound $C>0$ (independent of the time step number $l\geq 1$) for the value functions $u^{\rho,\nu,-,e,l}_i$, and for the function $\omega_i$ in original coordinates. Furthermore, there is an upper bound which is independent of the viscosity $\nu>0$ by construction and by the damping of the potential term.
Note that we get a similar upper bound for spatial derivatives $D^{\alpha}_xu^{\rho,\nu,-,e,l}$ inductively for $0\leq |\alpha|\leq m$ if the data $h_i$ have sufficient regularity
and $\max_{0\leq |\alpha|\leq m }{\big |}D^{\alpha}_xh_i(.){\big |}\leq C$. 

\vspace{0.5cm}

Next, concerning item iii) as we have regular data we can we observe that the contraction result holds with a contraction constant which is independent of the viscosity constant $\nu >0$. Hence we can use it later when we consider viscosity limits $\nu\downarrow 0$. For all $1\leq i\leq D$ we have the representation
\begin{equation}\label{transformeduu} 
\begin{array}{ll}
 u^{\rho,\nu,-,e}_{i}=u^{\rho,\nu,-}_i(0,.)\ast_{sp}G_{\nu}\\
\\
-\overline{\mu}^{\rho}\sum_jb^{\rho}_{j,t}
 u^{\rho,\nu,-,e}_{i,j}\ast G_{\nu}
 -\overline{\mu}^{\rho}\sum_{j}b^{\rho,s}_{j}u^{\int,\rho,\nu,-,e}_{j}u^{\rho,\nu,-,e}_{i,j}\ast G_{\nu}\\
 \\
+\overline{\mu}^{\rho}\sum_j\frac{1}{2}\left( (b^{\rho,s}_{j,j} u^{\int,\rho,\nu,-,e}_{i,j})+(b^{\rho,s}_{i,i} u^{\int,\rho,\nu,-,e}_{j,i})\right)u^{\rho,\nu,-,e}_j\ast G_{\nu}\\
\\
-\frac{1}{\rho^2}\frac{\sqrt{\rho^2-t^2(\sigma)}^3}{\rho-t}u^{\rho,\nu,-,e}_i\ast G_{\nu}.
\end{array}
\end{equation}
For multivariate spatial derivatives of order $1\leq |\alpha|\leq m$ and $|\beta|$ with $\beta_j+1=\alpha_j$ for some $1\leq j\leq D$ and $\beta_q=\alpha $ for $q\neq j$ we have
 \begin{equation}\label{transformeduuv}
 \begin{array}{ll}
 D^{\alpha}_xu^{\rho,\nu,-,e}_{i}=D^{\alpha}_xu^{\rho,\nu,-,e}_i(0,.)\ast_{sp}G_{\nu}\\
\\
-D^{\beta}_x\left( \overline{\mu}^{\rho}\sum_jb^{\rho}_{j,t}
 u^{\rho,\nu,-,e}_{i,j}\right) \ast G_{\nu,j}
 -D^{\beta}_x\left( \overline{\mu}^{\rho}\sum_{j}b^{\rho,s}_{j}u^{\int,\rho,\nu,-,e}_{j}u^{\rho,\nu,-,e}_{i,j}\right) \ast G_{\nu,j}\\
 \\
+D^{\beta}_x\left( \overline{\mu}^{\rho}\sum_j\frac{1}{2}\left( (b^{\rho,s}_{j,j} u^{\int,\rho,\nu,-,e}_{i,j})+(b^{\rho,s}_{i,i} u^{\int,\rho,\nu,-,e}_{j,i})\right)u^{\rho,\nu,-,e}_j\right)\ast G_{\nu,j}\\
\\
-\frac{1}{\rho^2}\frac{\sqrt{\rho^2-t^2(\sigma)}^3}{\rho-t}D^{\alpha}_xu^{\rho,\nu,-,e}_i\ast G_{\nu}.
\end{array}
\end{equation} 
The representation in (\ref{transformeduuv}) has the effect that the maximal order of derivatives is the same on the left and on the right side of the equation. For this reason such representations can be used fortime-local contraction results in strong spaces. 
Note that a natural estimate for the first order spatial derivatives of the Gaussian is
\begin{equation}
\begin{array}{ll}
{\big |}G_{\nu,j}(t,x;s,y){\big |}\leq \frac{|x_i-y_i|}{4\nu (t-s) \sqrt{4\nu(t-s)(t-s)}^3}\exp\left(-\frac{|x-y|^2}{4\nu(t-s)} \right)\\
\\
\leq \frac{1}{|x-y|}\frac{|x-y|^2}{4\nu (t-s)}\exp\left(-\frac{|x-y|^2}{8\nu(t-s)} \right)
\frac{1}{\sqrt{4\nu(t-s)(t-s)}^3}\exp\left(-\frac{|x-y|^2}{8\nu(t-s)} \right)\\
\\
\leq \frac{1}{|x-y|8\nu(t-s)} 
\end{array}
\end{equation}
where for $z_i=x_i-y_i$ and $z_i=\sqrt{\nu} z'_i$
\begin{equation}
\begin{array}{ll}
\int_{{\mathbb R}^D}\frac{1}{|z|}
\frac{C}{\sqrt{4\nu (t-s)}^3}\exp\left(-\frac{|z|^2}{8\nu(t-s)} \right)dz\\
\\
=\int_{{\mathbb R}^D}\frac{1}{\sqrt{\nu}|z'|}
\frac{C}{\sqrt{4 (t-s)}^3}\exp\left(-\frac{|z'|^2}{8(t-s)} \right)dz'\leq \frac{\tilde{C}}{\sqrt{\nu}|z'|(t-s)^{\delta}}
\end{array}
\end{equation}
for $\delta \in (0,1)$. 
Therefore we use symmetry of the first derivative of the Gaussian and Lipschitz regularity of the convoluted terms in the representation of $u^{\rho,\nu,-}_j,~1\leq j\leq D$ in (\ref{transformeduuv}) above, i.e. we consider convolutions with Lipschitz-continuous terms of the form
\begin{equation}
\begin{array}{ll}
h\in {\Big \{} D^{\beta}x{\Big (} \overline{\mu}^{\rho}\sum_jb^{\rho}_{j,t}
 u^{\rho,\nu,-}_{i,j}{\Big )},D^{\beta}_x{\Big (} \overline{\mu}^{\rho}\sum_{j}b^{\rho,s}_{j}u^{\int,\rho,\nu,}_{j}u^{\rho,\nu,-}_{i,j} {\Big )},\\
 \\
 \hspace{1cm} D^{\beta}_x{\Big (} \overline{\mu}^{\rho}\sum_j\frac{1}{2}\left( (b^{\rho,s}_{j,j} u^{\int,\rho,\nu,-}_{i,j})+(b^{\rho,s}_{i,i} u^{\int,\rho,\nu,-}_{j,i})\right)u^{\rho,\nu,-}_j{\Big )}{\Big \}}.
 \end{array}
\end{equation}
Note that a Lipschitz continuous $h$  implies that for all $x\in {\mathbb R}^D$ the function $h_x$ with
\begin{equation}
h_x(y)=h(x-y) \mbox{ for all }y\in {\mathbb R}^D
\end{equation}
is also Lipschitz. Then we have hier
\begin{equation}\label{visclim}
\begin{array}{ll}
 {\Big |}\int_{{\mathbb R}^D}h(y)\frac{2(x_i-y_i)}{4\nu s\sqrt{4\pi \nu s}^n}\exp\left(-\frac{|x-y|^2}{4\nu s}\right) dy{\Big |}\\
 \\
{\Big |}\int_{{\mathbb R}^D}h(x-y)\frac{2y_i}{4\nu s\sqrt{4\pi \nu s}^n}\exp\left(-\frac{|y|^2}{4\nu s}\right) dy{\Big |}\\ 
 \\
 \leq \int_{\left\lbrace y|y\in {\mathbb R}^D, y_i\geq 0\right\rbrace }{\Big |}h_x(y)-h_x(y^{i,-}){\Big |}\frac{|2y_i|}{4\nu s\sqrt{4\pi \nu s}^n}\exp\left(-\frac{|y|^2}{4\nu s}\right) dy\\
 \\
 \leq \int_{\left\lbrace y|y\in {\mathbb R}^D, y_i\geq 0\right\rbrace }L{\Big |}2y_i{\Big |}\frac{|2y_i|}{4\nu s\sqrt{4\pi \nu s}^n}\exp\left(-\frac{|y|^2}{4\nu s}\right) dy\\
 \\
 =\int_{\left\lbrace y'|y'\in {\mathbb R}^D, y'_i\geq 0\right\rbrace }L{\Big |}2\sqrt{\nu}y'_i{\Big |}\frac{2\sqrt{\nu}y'_i}{4\nu s\sqrt{4\pi  s}^n}\exp\left(-\frac{|y'|^2}{4 s}\right) dy'\\
 \\
 =\int_{\left\lbrace y'|y'\in {\mathbb R}^D, y'_i\geq 0\right\rbrace }L\frac{4(y'_i)^2}{4 s\sqrt{4\pi  s}^n}\exp\left(-\frac{|y'|^2}{4 s}\right) dy'
\end{array}
 \end{equation}
with $y=\sqrt{\nu}y'$ and for some Lipschitz constant $L$. Here, $y^{i,-}=(y^{i,-}_1,\cdots,y^{i,-}_n)$
\begin{equation}
y^{i,-}_j=y^{i}_j,~\mbox{ if }j\neq i~\mbox{ and }~y^{i,-}_i=y_i~\mbox{ for }~j=i.
\end{equation}

Note that the left side of (\ref{visclim}) is independent of $\nu$, because the Lipschitz constant can be chosen independently of $\nu$ for the class of data considered in item iv). This is shown below.

Hence, we have
\begin{cor}
For the class of data $h_i$ considered in item iv)
the contraction result in theorem \ref{contrthm} for the functions $u^{\rho,\nu,-,e,l}_i$ hold in the viscosity limit $\nu\downarrow 0$.
\end{cor}
It follows that the equation for the function $u^{\rho,\nu,-,e,l}_i,~1\leq i\leq D$ is satisfied for all $\nu >0$ with a uniform bound on the Laplacian which vanishes in the limit $\nu\downarrow 0$ as we have the contraction space $H^m\cap C^m$ for $m\geq  2$. Hence the viscosity limit function $u^{\rho,\nu,-,e}_i,~ 1\leq i\leq D$ is a regular factor of a local classical solution of the incompressible Euler equation in local original time $t\in [0,\rho)$. 
\vspace{0.5cm}

Finally, ad iv), as we have a local solution $u^{\rho,\nu,-,*,e}_i,~1\leq i\leq D$ corresponding to a global solution $u^{\rho,\nu,-,e}_i,~1\leq i\leq D$  with upper bound
\begin{equation}
\sup_{\sigma\in [0,\infty)}{\big |}u^{\rho,\nu,-,e}_i(\sigma,.){\big |}_{H^m\cap C^m}\leq C
\end{equation}
for some finite constant $C>0$, we may establish upper bounds for the increments $\delta u^{\rho,\nu,-,e}_i,~1\leq i\leq D$ defined for time $s\in [0,\infty)$ or for related increments $\delta u^{\rho,\nu,-,*,e}_i,~1\leq i\leq D$ defined on the corresponding interval $[0,\rho)$. Assume that
\begin{equation}
u^{\rho,\nu,-,e}_i=u^{\rho,\nu,-,*,e}_i(0,0)=c\neq 0,
\end{equation}
for some $1\leq i\leq D$ and that $\nu$ is small.
We show that for the regular class of data introduced at item iv) above,  upper bounds of the increments  $\delta u^{\rho,\nu,-,*,e}_i,~1\leq i\leq D$ are small enough such that for some $1\leq i\leq D$
\begin{equation}
{\big |} u^{\rho,\nu,-,*,e}_i(\rho,0){\big |}\geq 0.5 c>0,~\mbox{for some }~\rho >0,
\end{equation}
or 
\begin{equation}
{\big |}\delta u^{\rho,\nu,-,*,e}_i(\rho,0){\big |}\leq 0.5 c>0,~\mbox{for some }~\rho >0,
\end{equation}
where $c>0$ is the modulus of $u^{\rho,\nu,-,*}_i$ at the origin.
Here, recall that we adjusted the coreesponding viscosity parameter $\nu'$ in order to have a simple viscosity extension of the time transformed equation. We have  $u^{\rho,\nu,-,*,e}_{i}(0,.)=h^{\rho,e}_i(.)$, where the later function denotes the transformed data trivially extended to the whole of ${\mathbb R}^D$. Hence we have the value function representation
\begin{equation}\label{transformed2*end}
\begin{array}{ll}
u^{\rho,\nu,-,*,e}_{i}(t,y)=\left( h^{\rho,e}_{i}(.)\ast_{sp}G^*_{\nu}\right)(t,y)\\
\\
-\left( \sum_jb^{\rho}_{j,t}u^{\rho,\nu,-,*,e}_{i,j}\ast G^*_{\nu}\right)(t,y)\\
\\
-\int_{D_{\rho}}\left( \sum_{j}b^{\rho,s}_{j,j}u^{\int,\rho,\nu,-,*,e}_{j}u^{\rho,\nu,-,*,e}_{i,j}\right) (\tau,z) G^*_{\nu}(t,y;\tau,z)dzd\tau\\
\\
-\int_{D_{\rho}}\frac{1}{\rho-\tau}u^{\rho,\nu,-,*,e}_i(s,z)G^*_{\nu}(t,y;\tau,z)dzd\tau\\
 \\
+\int_{D_{\rho}}\left( \sum_j\frac{1}{2}\left( b^{\rho,s}_{j,j} u^{\int,\rho,\nu,*,e}_{i,j}+b^{\rho,s}_{i,i} u^{\int,\rho,\nu,-,*,e}_{j,i}\right)u^{\rho,\nu,-,*,e}_j\right) (\tau,z)G^*_{\nu}(t,y;\tau,z)dzd\tau.
\end{array}
\end{equation}
Hence, for the increments $\delta u^{\rho,\nu,-,*,e}_i=u^{\rho,\nu,-,*,e}_{i}-h^e_i$ evaluated at $(\rho,0)$ we have the classical representation
\begin{equation}\label{transformedrep}
\begin{array}{ll}
\delta u^{\rho,\nu,-,*,e}_{i}(\rho,0)=
-\lim_{\tau\uparrow \rho}\int_{D_{\rho}}\sum_jb^{\rho}_{j,t}
 u^{\rho,\nu,-,*,e}_{i,j}G^*_{\nu}(t,0;\tau,z)dzd\tau\\
 \\
 -\lim_{\tau\uparrow \rho}\int_{D_{\rho}}\left( \sum_{j}b^{\rho,s}_{j,j}u^{\int,\rho,\nu,-,*,e}_{j}u^{\rho,\nu,-,*,e}_{i,j}+\frac{1}{\rho-\tau}u^{\rho,\nu,-,*,e}_i\right)\times\\
 \times G^*_{\nu}(t,0;\tau,z)dzd\tau\\
 \\
-\lim_{\tau\uparrow \rho}\int_{D_{\rho}}\left(\sum_j\frac{1}{2}\left( (b^{\rho,s}_{j,j} u^{\int,\rho,\nu,-,*,e}_{i,j})+(b^{\rho,s}_{i,i} u^{\int,\rho,\nu,-,*,e}_{j,i})\right)u^{\rho,\nu,-,*,e}_j\right)\times\\
\times G^*_{\nu}(t,0;\tau,z)dzd\tau\\
\\
+\lim_{t\uparrow \rho}\int_{D_{\rho}}h^{\rho,e}_i(y)G^*_{\nu}(t,0;0,z)dz-h^{\rho,e}_i(y).
\end{array}
\end{equation}
As $\nu>0$ becomes small the main  mass of this increment mainly is concentrated on the cone $K^{\rho,t}_0$ with exponential spatial decay outside the cone (where in a viscosity limit $\nu\downarrow 0$ the domain $D_{\rho}$ can be substituted by $K^{\rho,t}_0$. More precisely for any given $\epsilon >0$ for  $\nu >0$ small enough we have
\begin{equation}\label{transformedrep2}
\begin{array}{ll}
|\delta u^{\rho,\nu,-,*,e}_{i}(\rho,0)|\leq 
{\big |}\lim_{\tau\uparrow \rho}\int_{K^{\rho,t}_0}\sum_jb^{\rho}_{j,t}
 u^{\rho,\nu,-,*,e}_{i,j}G^*_{\nu}(t,0;\tau,z)dzd\tau{\big |}\\
 \\
 +\lim_{\tau\uparrow \rho}\int_{K^{\rho,t}_0}{\big |}\left( \sum_{j}b^{\rho,s}_{j,j}u^{\int,\rho,\nu,-,*,e}_{j}u^{\rho,\nu,-,*,e}_{i,j}{\big |}+{\Big |}\frac{1}{\rho-\tau}u^{\rho,\nu,-,*,e}_i\right){\Big |}\times\\
 \times G^*_{\nu}(t,0;\tau,z)dzd\tau\\
 \\
+\lim_{\tau\uparrow \rho}{\Big |}\int_{K^{\rho,t}_0}\left(\sum_j\frac{1}{2}\left( (b^{\rho,s}_{j,j} u^{\int,\rho,\nu,-,*,e}_{i,j})+(b^{\rho,s}_{i,i} u^{\int,\rho,\nu,-,*,e}_{j,i})\right)u^{\rho,\nu,-,*,e}_j\right)\times\\
\times G^*_{\nu}(t,0;\tau,z)dzd\tau{\Big |}\\
\\
+\lim_{\tau\uparrow \rho}{\Big |}\int_{K^{\rho,t}_0}h^e_i(y)G^*_{\nu}(t,0;0,z)dz-h^{\rho,e}_i(y){\Big |}+\epsilon.
\end{array}
\end{equation}

The set of regular data $h^{\rho,e}_i,~1\leq i\leq D$ (which is even dense in $L^2$ for each component $1\leq i\leq D$) may be chosen as in item iv) above, i.e., we consider the set of bounded regular functions $g_i,~1\leq i\leq D$ with bounded derivatives multiplied by a factor of polynomial decay of order $10$, i.e., of the form $h^{\rho,e}_i=\frac{g_i}{1+|x|^{10}}$,
where $g_i$ are smooth functions with bounded derivatives (for example finite Fourier sums). We consider a variation of the iteration scheme above:
for an iteration index $q\geq 1$  the functions $u^{\rho,\nu,-,*,e,q}_{i},~1\leq i\leq D,~q\geq 1$ are determined by $u^{\rho,\nu,-,*,e,0}_i=h^{\rho,e}_i(0,.)$ for $q= 0$ and, recursively, for $q+1\geq 1$ by solutions of linearized equations of the form
\begin{equation}\label{transformediter}
\begin{array}{ll}
u^{\rho,\nu,-,*,e,q+1}_{i,t}-\nu \Delta u^{\rho,\nu,-,*,e,q+1}_{i}\\
\\
+\sum_jb^{\rho}_{j,t}
 u^{\rho,\nu,-,*,e,q+1}_{i,j}
 +\sum_{j}b^{\rho,s}_{j,j}u^{\int,\rho,\nu,-,*,e,q}_{j}u^{\rho,\nu,-,*,e,q}_{i,j}\\
 \\
=\sum_j\frac{1}{2}\left( (b^{\rho,s}_{j,j} u^{\int,\rho,\nu,-,*,e,q}_{i,j})+(b^{\rho,s}_{i,i} u^{\int,\rho,\nu,-,*,e,q}_{j,i})\right)u^{\rho,\nu,-,*,e,q}_j\\
\\
-\frac{1}{\rho-t}u^{\rho,\nu,-,*,e,q+1}_i.
\end{array}
\end{equation}
Note that the first order term in (\ref{transformediter}), i.e., the term $\sum_jb^{\rho}_{j,t}u^{\rho,\nu,-,q+1}_{i,j}$ has the iteration index $q+1$. Similar for the last damping term. They are part of a fundamental solution in a classical representation with product terms.  The idea here is that the representation of functional increments in form of  a sum of integrals with integrated product terms is a well-suited set up in order to get some preservation of polynomial decay. For these classical representations we use
the fundamental solution $p^{\rho,\nu}$ of the equation
\begin{equation}
p^{\rho,\nu}_{i,t}-\nu \Delta p^{\rho,\nu}+\sum_jb^{\rho}_{j,t} p^{\rho,\nu}_{,j}=
\frac{1}{\rho-t}p^{\rho,\nu}.
\end{equation}
In the following the symbol $C^{m,n}_{p,pol_k}$ denotes the space of functions with continuous time derivatives up to order $m$ and continuous bounded spatial derivatives up to order $n$ where the value function itself and all the derivatives are all of spatial polynomial decay of order $k$.
We may estimate the increment in original time $t$-coordinates. Furthermore we shall transform the cone to a cylinder in order to estimate the potential damping term using the observation in (\ref{transformedrep2}). We claim that
\begin{itemize}
 \item[iv)a)] for each iteration number $q$ the functional increments components  $$\delta u^{\rho,\nu,-,*,e,q}_i,~1\leq i\leq D$$ solving (\ref{transformediter})  starting with data $h^e_i$ of the regular data set of item iv) are in $C^{1,2}_{b,pol_{13}}$. The component functions $u^{\rho,\nu,-,*,e,q}_i,~1\leq i\leq D$ itself are in $C^{1,2}_{b,pol_{10}}$.
 \item[iv)b)] the last claim of item iv)a) holds in the limit, i.e. for the functions
 \begin{equation}
u^{\rho,\nu,*,e,-}_i(t,.)= h^{\rho,e}_i(.)\ast_{sp}p^{\rho,\nu}+\sum_{q\geq 1}\delta u^{\rho,\nu,-,*,e,q}_i(t,.),~1\leq i\leq D
 \end{equation}
Furthermore for $\nu >0$ small enough the increments $\delta u^{\rho,\nu,-,*}_i(t,.)$ and $\delta u^{\rho,\nu,-,*}_i(\sigma,.)$ decrease with $\rho$ such that for $\rho >0$ small enough we have especially
\begin{equation}\label{aberaber}
{\big |}u^{\rho,\nu,-,*,e}_i(\rho,0)- h^{\rho,e}_i(0){\big |}\leq 0.5 c ,
\end{equation}
if ${\big |}h^{\rho,e}_i(0){\big |}\geq c$ for some $c>0$.
\item[iv)c)] there is a subsequence $(\nu_k)_{k\geq 1}$and a viscosity limit $\lim_{k\downarrow 0}u^{\rho,\nu_k,-,*,e}_i,~ 1\leq i\leq D$ (encoding the regular part of the vorticity solution of the incompressible Euler equation), such that (\ref{aberaber}) holds for this viscosity limit.
\end{itemize}

Next we consider the three subclaims more closely.
Ad iv)a), from (\ref{transformediter}) we observe that for $q\geq 1$ the function $u^{\rho,\nu,-,*,e,q+1}_{i},~1\leq i\leq D$ has the representation

\begin{equation}\label{transformediterb}
\begin{array}{ll}
u^{\rho,\nu,-,*,e,q+1}_{i}=h^{\rho,e}_i\ast_{sp}p^{\rho}_{\nu}-{\Big (}\sum_{j}b^{\rho,s}_{j,j}u^{\int,\rho,\nu,-,*,e,q}_{j}u^{\rho,\nu,-,*,e,q}_{i,j}{\Big )}\ast p^{\rho,\nu}\\
 \\
+{\Big (}\sum_j\frac{1}{2}\left( (b^{\rho,s}_{j,j} u^{\int,\rho,\nu,-,*,e,q}_{i,j})+(b^{\rho,s}_{i,i} u^{\int,\rho,\nu,-,*,e,q}_{j,i})\right)u^{\rho,\nu,-,*,e,q}_j{\Big )}\ast p^{\rho,\nu}
\end{array}
\end{equation}
where $p^{\rho,\nu}$ is the fundamental solution defined above. Hence,
\begin{equation}\label{transformediterc}
\begin{array}{ll}
\delta u^{\rho,\nu,-,*,e,q+1}_{i}=u^{\rho,\nu,-,q+1}_{i}-u^{\rho,\nu,-,*,e,q}_{i}={\Big (}\sum_{j}b^{\rho,s}_{j,j}u^{\int,\rho,\nu,-,*,e,q}_{j}u^{\rho,\nu,-,*,e,q}_{i,j}{\Big )}\ast p^{\rho,\nu}\\
 \\
+{\Big (}\sum_j\frac{1}{2}\left( (b^{\rho,s}_{j,j} u^{\int,\rho,\nu,-,*,e,q}_{i,j})+(b^{\rho,s}_{i,i} u^{\int,\rho,\nu,-,*,e,q}_{j,i})\right)u^{\rho,\nu,-,*,e,q}_j{\Big )}\ast p^{\rho,\nu}\\
\\
-{\Big (}\sum_{j}b^{\rho,s}_{j,j}u^{\int,\rho,\nu,-,*,e,q-1}_{j}u^{\rho,\nu,-,*,e,q-1}_{i,j}{\Big )}\ast p^{\rho,\nu}\\
 \\
-{\Big (}\sum_j\frac{1}{2}\left( (b^{\rho,s}_{j,j} u^{\int,\rho,\nu,-,*,e,q-1}_{i,j})+(b^{\rho,s}_{i,i} u^{\int,\rho,\nu,-,*,e,q-1}_{j,i})\right)u^{\rho,\nu,-,*,e,q-1}_j{\Big )}\ast p^{\rho,\nu}.
\end{array}
\end{equation}


It is essential to observe the preservation of polynomial decay for the Leray projection term. The considerations for the Burgers term are similar (somewhat simpler).   We have a closer look at the Leray projection term.  Intersecting terms of the form
terms of the for $$\sum_j\frac{1}{2}\left( (b^{\rho,s}_{j,j} u^{\int,\rho,\nu,-,*,e,,q}_{i,j})+(b^{\rho,s}_{i,i} u^{\int,\rho,\nu,-,*,e,q}_{j,i})\right)u^{\rho,\nu,-,*,e,q}_j$$
the Leray projection terms on the right side of (\ref{transformediterc}) are of the form
\begin{equation}\label{Leray1}
\sum_j\frac{1}{2}\left( (b^{\rho,s}_{j,j} u^{\int,\rho,\nu,-,*,e,p}_{i,j})+(b^{\rho,s}_{i,i} u^{\int,\rho,\nu,-,p}_{j,i})\right)\delta u^{\rho,\nu,-,*,e,p}_j,
\end{equation}
and
\begin{equation}\label{Leray2}
\sum_j\frac{1}{2}\left( (b^{\rho,s}_{j,j} \delta u^{\int,\rho,\nu,-,*,e,p}_{i,j})+(b^{\rho,s}_{i,i} \delta u^{\int,\rho,\nu,-,*,e,p}_{j,i})\right) u^{\rho,\nu,-,e,p}_j,
\end{equation}
for iteration indices $p$. Here $b^{s}_j=\frac{2}{\pi}\arctan(x_j)$ such that in (\ref{Leray2}) we have
\begin{equation}\label{factor1}
(b^{\rho,s}_{j,j} \delta u^{\int,\rho,\nu,-,*,e,p}_{i,j})=
\int_{D_{\rho}}K^*_3(w)u^{\rho,\nu,-,*,e,p}_{i,j}(\sigma,y-w)\frac{\pi^3}{2^3}
\frac{\Pi_{i=1}^3(1+x_i^2)}{\left(\rho-t\right)^3}dw,
\end{equation}
where the latter term has its essentially mass on $K^{\rho,t}_0$ in the sense of a statement as at  (\ref{transformedrep2}) above.
It suffices to estimate convolutions of the Leray projection terms (\ref{Leray1}) and (\ref{Leray2}) with Gaussian upper bounds of the fundamental solution $p^{\rho,\nu}$.
As we have a damping term we have certainly such a Gaussian upper bound
\begin{equation}\label{gaussup}
\begin{array}{ll}
{\big |}p^{\nu,\rho}(t,x;s,y){\Big |}\leq \frac{C}{\sqrt{\pi (t-s}^D}\exp\left(-\frac{\lambda |x-y|^2}{4(t-s)} \right) 
\end{array}
\end{equation}
for some $\lambda>0$ and $C>0$. Rough standard convolution estimates of the
form (where we may assume that $q,p\geq D$)  
\begin{equation}
\int \frac{c_1}{(1+|x-y|)^p(1+|y|)^q}\leq \frac{c_2}{(1+|x|)^{p+q-D}}
\end{equation}
(where we may assume that $q,p\geq D$) that inductively we get polynomial decay of order $13-6-1$ for the factor in (\ref{Leray2}). The factor $u^{\rho,\nu,-,*,e,p}_j$ in (\ref{Leray2}) is inductively of order $10$ and the convolution with the Gaussian decreases the polynomial deal by no more than order $3$ (all considerations in dimension $D=3$) such that the (modulus of the ) term in (\ref{Leray2}) convoluted with the Gaussian upper bound of  $p^{\rho,\nu}$ has polynomial decay of order $10$. Similarly for the other incremental terms and for their spatial derivatives. The functions $u^{\rho,\nu,-,*,e,q}_i,~1\leq i\leq D$ themselves have additional terms involving initial data such that we get polynomial decay of order $10$ for these component functions inductively.

Ad iv)b), we first remark that the claim of iva) holds also in the limit $q\uparrow \infty$ because of the contraction  result above. Furthermore the contraction constant can be chosen independently of the viscosity $\nu>0$, because of Lipshitz continuity of the convoluted data in propbabilistic represenations of the solutions. Next recall that for $\nu>0$ the contribution of integrals in the probabilistic representations of the increment $|\delta u^{\rho,\nu,-,*,e}_{i}(\rho,.)|$ becomes small outside the cone $K^{\rho,t}_0$ as $\nu>0$ beomes small. We observed that for $\nu >0$ small enough for all $\epsilon >0$ there exists a $\rho >0$ such that
\begin{equation}\label{transformedrep2end}
\begin{array}{ll}
|\delta u^{\rho,\nu,-,*,e}_{i}(\rho,0)|\leq 
{\big |}\lim_{\tau\uparrow \rho}\int_{K^{\rho,t}_0}\sum_jb^{\rho}_{j,t}
 u^{\rho,\nu,-,*,e}_{i,j}G^*_{\nu}(t,0;\tau,z)dzd\tau{\big |}\\
 \\
 +\lim_{\tau\uparrow \rho}\int_{K^{\rho,t}_0}{\big |} \sum_{j}b^{\rho,s}_{j,j}u^{\int,\rho,\nu,-,*,e}_{j}u^{\rho,\nu,-,*,e}_{i,j}{\big |}G^*_{\nu}(t,0;\tau,z)dzd\tau\\
 \\
+\lim_{\tau\uparrow \rho}{\Big |}\int_{K^{\rho,t}_0}\left(\sum_j\frac{1}{2}\left( (b^{\rho,s}_{j,j} u^{\int,\rho,\nu,-,*,e}_{i,j})+(b^{\rho,s}_{i,i} u^{\int,\rho,\nu,-,*,e}_{j,i})\right)u^{\rho,\nu,-,*,e}_j\right)\times\\
\times G^*_{\nu}(t,0;\tau,z)dzd\tau{\Big |}\\
\\
+\lim_{\tau\uparrow \rho}{\Big |}\int_{K^{\rho,t}_0}h^{\rho,e}_i(y)G^*_{\nu}(t,0;0,z)dz-h^{\rho,e}_i(y){\Big |}\\
\\
+\lim_{\tau\uparrow \rho}\int_{K^{\rho,t}_0}{\Big |}\frac{1}{\rho-\tau}u^{\rho,\nu,-,*,e}_i{\Big |}G^*_{\nu}(t,0;\tau,z)dzd\tau+\epsilon.
\end{array}
\end{equation}
Here the cone $K^{\rho,t}_0$ is the cone in original coordinates where the time $t$ is in the interval $[0,\rho]$. The related integrals are all over the time interval $[0,\rho]$ with bounded integrands. Note that these integrands have upper bounds independent of $\nu$ and $\rho$., and note that for the volume of the cone we have
\begin{equation}
\mbox{vol}\left(K^{\rho,t}_0 \right) \sim \rho^D.
\end{equation}
\begin{rem}
For data $h^{\rho,e}_i,~ 1\leq i\leq D$ with $\max_{1\leq i\leq D}{\big |}h^{\rho,e}_i{\big |}_{H^m\cap C^m}=C$ the approopriate choice of $\rho$ depends clearly on the size of $C$. For small $C$ for some $m\geq 1$ the increment of the nonlinear terms is small on a small time interval $\rho$. However, even for 
However there also data where the first order derivatives are much larger then the data itself,and for such data we cannot find a sufficiently small upper bound of the increment. Hence, we claim only the existence of a dense set of data in $L^2$-sense which lead to a singularity after finite time.  
\end{rem}  
Hence, all terms go obviously to zero as $\rho\downarrow 0$ except for the damping term for which this seems less obvious. However, we have
\begin{equation}
\begin{array}{ll}
{\Bigg |}\int_{K^{\rho,t}_0} \frac{1}{\rho-s_0}Cdyds_0{\Bigg |}\leq \frac{8}{3}C\rho^3,
\end{array}
\end{equation}
where we may use the spatial transformation $dw_i=\frac{dy_i}{\rho-s_0}$ from the cone to a cylinder.
\begin{rem}
Note that in the last step we could also integrate with respect to time first 
\begin{equation}
\begin{array}{ll}
{\Bigg |}\int_{K^{\rho,t}_0} \frac{1}{\rho-s_0}Cdyds_0{\Bigg |}= \lim_{t\uparrow \rho}{\Bigg |}\int_{0}^t\int_{|y|\leq (\rho-s_0)} \frac{1}{\rho-s_0}Cdyds_0{\Bigg |}\\
\\
\leq \lim_{t\uparrow \rho}{\Bigg |}\int_{|y|\leq (\rho-t)} (\ln(\rho-t)-\ln(\rho)Cdyds_0{\Bigg |}\downarrow 0 \mbox{ as }\rho\downarrow 0.
\end{array}
\end{equation}
\end{rem}
Hence if $u^{\rho,\nu,-,*,e}_i(0,0)=c\neq 0$ for some $1\leq i\leq 3$, and if the data $h^e_i,~1\leq i\leq D$ are regular of the form as in item iv) above, then there is a small time horizon $\rho>0$ such that for all $\nu >0$
\begin{equation}
{\big |}\delta u^{\rho,\nu,-,*,e,}_i(\rho,0){\big |}\leq \frac{c}{2}, 
\end{equation}
such that
\begin{equation}
{\big |}\delta u^{\rho,\nu,-,*,e}_i(\rho,0){\big |}\geq \frac{c}{2}>0,
\end{equation}
and by the argument of item iii) and (using the regularity of the convoluted terms as shown above in this item by an iteration argument) this follows also for the viscosity limit $\nu \downarrow 0$, and, hence, it follows that the vorticity $\omega_i$ has a singular point at $(\rho,0)$.

Finally, ad iv)c) we remark that the strang polynomial decay allows us to obtain viscosity limits by using compactness arguments on compact spaces via spatial transformations using  classic function spaces $C^{m,n}$ for some $m,n$ compatible with the strength of polynomial decay (we need 2 order of polynomial decay for each spatial regularity order of the viscosity limit -we have considered this elsewhere).

The local analysis can be transferred to the compressible Navier Stokes equation in (\ref{navcomp}) considering inviscid limits $\mu ,\mu'\downarrow 0$. The extension is  quite straightforward on a domain of real degeneracies (in the sense explained above) in the region of incompressibility, i.e., in the region $D_e\cap \left\lbrace (t,x)| \mbox{div}v(t,x)=0\right\rbrace $ (if  this subset of real degeneracies is not empty). Otherwise the scalar densities in the local analysis above have to be replaced by fundamental solutions of linear parabolic systems which are coupled iteratively with transport equations of first order. The extension of the local analysis is then straightforward if it is ensured that the mass density is strictly positive almost surely. If the data are all strictly positive this seems to be true, but this will be considered elsewhere.

Next, arguing for a) of the preceding section we first remark that an attitude towards singularities which may be dubbed horror singulari, may often be a hindrance in studying certain classes of solutions which may be otherwise an important guide for further analysis- such as the classical singular solutions of the Einstein field equations may be an important guide for further investigation of gravity. A similar point of view with respect to singular solutions of equations in fluid mechanics seems appropriate. Sometimes it is thought that the solution is only local if there is a singularity. However, even the function $t\rightarrow \frac{x_0}{1-t}$ is a solution of the ODE $\stackrel{\cdot}{x}=x^2,~x(0)=x_0=1$ 
on the domain $\left(0,\infty\right) \setminus \left\lbrace  1\right\rbrace$, although we loose uniqueness after time $t=1$. Singularities may be artificial but not more then other aspects of our description of nature to our best knowledge. On the other hand the existence of singularities is a rather weak insight compared to what may be expected of a quantitative theory of turbulence. Next we argue that the existence of a vorticity singularity as constructed is a weak concept of turbulence in the sense that local solutions in the neighborhood of such a singularity exemplify the mentioned list of notions of turbulence in \cite{LL}. Assume that we have a local vorticity solution $\omega_i(t,y)=\frac{f_i(t,y)}{1-t}$ on a domain $[t_i,t_e]\times {\mathbb R}^3$ with a regular function $f_i$ which becomes singular at $(1,y_0)$ as $f_i(t_0,y_0)\neq 0$ for $1\in (t_i,t_e)$ for some $1\leq i\leq 3$. Then $\omega(.,y_0)$ changes sign at $t=1$ which implies extreme changes of the velocity magnitude in the neighborhood of the singularity compatible with notion $\alpha)$ of the description above. Second, the velocity integral $v_i=\int K_3\omega_i$ may imply that although we have extreme changes of velocity across $t=1$ we may have $v_i(1,y_0)=0$, i.e., a stationary point at $(1,y_0)$ and this is compatible with $\beta)$ of that list. Third, as there may be a stationary point of velocity zero at the singular point of vorticity $(1,y_0)$ the amplitudes of the vorticity may be large compared to the velocity itself, and this is compatible with $\gamma)$ of Landau's list. Fourth, even as we freeze time time $t_1$ close to $t=1$ small changes of $f_i(t_1,y)$ can lead to large changes with proportionality factor $\frac{1}{1-t_1}$ of the vorticity, which is compatible with $\delta)$ of the list above. Finally the extreme changes  of velocity near time $t=1$ is compatible with complicated trajectories and strong mixture as noted in $\epsilon)$ above. The latter point exemplifies the weakness of the singularity concept of turbulence, as the notion of strong mixture and related notions can be studied quantitatively by bifurcation theory. Finally we mention that the technique introduced in this article can be used to show that Navier Stokes equation models with $L^2$-force term can have singular solutions. 

\footnotetext[1]{\texttt{{kampen@mathalgorithm.de}, {kampen@wias-berlin.de}}.}
 
\begin{rem}
There has been a lot of research on singularity criteria of solutions for the Euler equation. The preceding text is an independent research work with a result which has not be covered elsewhere. As a non-specialist of the Euler equation I do not have the competence to choose and comment the huge amount of literature which has a content which is beyond that of the book of Majda and Bertozzi cited in \cite{MB} below.
\end{rem}

\end{document}